\documentclass[reqno,a4paper]{article}
\usepackage{amssymb,amsmath,epsfig,graphics,graphicx,color,subfigure}

\definecolor{lightblue}{rgb}{0.22,0.45,0.70}
\usepackage[colorlinks=true,breaklinks=true,linkcolor=lightblue,citecolor=lightblue]{hyperref}
\usepackage{float}
\usepackage{comment}

\newtheorem{remark}{Remark}[section]
\newtheorem{lemma}{Lemma}[section]
\newtheorem{theorem}{Theorem}[section]
\newtheorem{proposition}{Proposition}[section]

\newtheorem{assumption}{Assumption}

\usepackage{soul}

\renewcommand{\O}{\Omega}

\newcommand{\calA}{ \mathcal{A}}

\newcommand{\calB}{ \mathcal{B}}

\renewcommand{\P}{{\mathcal P}}  % polynomials

\def\Eta{\boldsymbol{\eta}}

\def\nbs \boldsymbol{n}
\def\0{\boldsymbol{0}}

\def\ZtK{\widetilde{\mathcal{Z}_h^{E}}}
\def\ZK{\mathcal{Z}_h^{E}}
\def\dof{\text{dof}}
\def\TAstr{\widetilde{\alpha_{\ast}}}
\def\TAstR{\widetilde{\alpha^{\ast}}}

\def\TGstR{\widetilde{\gamma^{\ast}}}
\def\TGstr{\widetilde{\gamma_{\ast}}}

\def\Hdost{{H^2_{\ast}(\O)}}

\def\PiK{\Pi_{E}^{k,\Delta}}

\def\aL{\widetilde{\alpha_{\ast}}}
\def\aU{\widetilde{\alpha^{\ast}}}

\def\gmL{\widetilde{\gamma_{\ast}}}
\def\gmU{\widetilde{\gamma^{\ast}}}

\newenvironment{proof}{\noindent{\it Proof.}}{\hfill$\square$}

\begin{document}

%***********************************************************************************
\title{Convergence Analysis of Virtual Element Method for Nonlinear Nonlocal Dynamic Plate Equation}
%***********************************************************************************

\author{D. Adak\thanks{GIMNAP, Departamento de Matem\'atica, Universidad
del B\'io-B\'io, Concepci\'on, Chile.
E-mail: {\tt dadak@ubiobio.cl}},\quad 
D. Mora\thanks{GIMNAP, Departamento de Matem\'atica, Universidad
del B\'io-B\'io, Concepci\'on, Chile and
CI$^2$MA, Universidad de Concepci\'on, Concepci\'on, Chile.
E-mail: {\tt dmora@ubiobio.cl}},
\quad
S. Natarajan\thanks{Integrated Modelling and Simulation Lab,
Department of Mechanical Engineering, Indian Institute
of Technology Madras, Chennai-600036, India.
E-mail: {\tt snatarajan@iitm.ac.in}}.
}

\date{}
\maketitle

\begin{abstract}
In this article, we have considered a nonlinear nonlocal
time dependent fourth order equation demonstrating the deformation
of a thin and narrow rectangular plate. We propose $C^1$ conforming virtual element
method (VEM) of arbitrary order, $k\ge2$, to approximate the model problem numerically.
We employ VEM to discretize the space variable and fully implicit
scheme for temporal variable. Well-posedness of the fully discrete
scheme is proved  under certain conditions on the physical parameters,
and we derive optimal order of convergence
in both space and time variable. Finally, numerical
experiments are presented to illustrate the behaviour
of the proposed numerical scheme. 
\end{abstract}

\medskip

\medskip

\section{Introduction}

We consider the nonlinear nonlocal time dependent equation which demonstrates
the deformation of a thin and narrow rectangular plate having two free
long edges and two hinged short edges. In neutral, the plate lies
horizontally flat representing a real suspension bridge with
planar computational domain. The plate is dealt with internal
(due to its own weight) force and external force (due to the load of vehicles and people)
that act orthogonally downward which are neutralised by the compressive
forces along the edges, the so-called buckling loads. Further,
we consider here the simplest scenario neglecting the complex interactions
between all the components of a real bridge.
The planar domain $\Omega$ is represented by $\Omega:=(0,L)\times(-\ell,\ell)$
with $0< \ell \ll L$.  Following \cite{BGMsD-SIAP2019,FGMsD-JDE2016,GazzoBook2015},
the nonlinear nonlocal dynamic plate equation which models the deformation
of bridge is given by
\begin{align}
& D_{tt}u +\delta D_t u+\Delta^2 u+\left[P-S \int_{\Omega} (D_xu)^2 {\rm d} \O \right]D_{xx}u=g
\quad \text{in} \ \Omega_T:=\Omega \times [0,T],\label{modl_prob:1} \\
& u=D_{xx}u=0 \quad \qquad \qquad \qquad \qquad \qquad\qquad \qquad \qquad \qquad \qquad  \text{on}\ \Gamma_{s} \times [0,T], \label{modl_prob:2}\\
& D_{yy}u+\sigma D_{xx}u=D_{yyy}u+(2-\sigma)D_{xxy}u=0 \qquad  \qquad \qquad \quad \; \text{on}\ \Gamma_{f} \times [0,T],  \label{modl_prob:3}\\
& u(x,y,0)=u_0(x,y), \quad
D_tu(x,y,0)=\omega_0(x,y) \qquad \qquad \qquad \qquad   \text{in}\ \Omega, \label{modl_prob:5}
\end{align}
where $D_t u$ and $D_x u$ denote the derivative of a function $u$
with respect to time variable $t$ and space variable $x$ and $\Gamma_{s}:=\{0,L\} \times [-\ell,\ell]$, and $\Gamma_{f}:=[0,L] \times \{-\ell,\ell\}$.
The constants appearing in the model problem~\eqref{modl_prob:1}--\eqref{modl_prob:5}
are expressed as below:
\begin{itemize}
\item $L=$ length of the plate;
\item $2\ell=$ width of the plate;
\item $\delta=$ damping coefficient;
\item $\sigma=$ Poisson ratio of the material of the plate; 
\item $P=$ pre-stressing constant;
\item $S>0$ is a coefficient that depends on the elasticity of the material composing the plate;
\item $g$ represents the vertical load over the plate and may
depend on space and time.
\end{itemize}

The deformation of the plate
is described by the function $u(x,y,t)$.
The model problem consists of nonlocal nonlinearity involving the buckling constant $P$
for which we have that $P>0$ if the plate is compressed and $P<0$
if the plate is elongated in the $x$ direction.
The term $S\int_{\O}(D_xu)^2{\rm d} \O$ measures the geometric nonlinearity
of the plate due to its stretching.
Moreover, the term $[P-S \int_{\Omega} (D_xu)^2{\rm d} \O]$ carries a nonlocal effect
into the model.
We refer to \cite{BGMsD-SIAP2019,FGMsD-JDE2016} for further details.
In \cite{BGMsD-SIAP2019}, the well-posedness  of the model
problem~\eqref{modl_prob:1}--\eqref{modl_prob:5} was analyzed. 
Recently, in \cite{CCCHSV-jFI2020} it has been studied the uniform decay rates in the presence
of nonlocal nonlinearity due to applying external forces. Moreover, a finite difference scheme
for a linearised scheme is proposed to validate the theoretical results.

In this article, we exploit a conforming $C^1$ virtual element to approximate the
solution of the nonlinear plate equation. The model problem involves fourth order
derivative in space variable which a conforming discretization requires globally $C^1$ functions.
It is well-known that the construction of $H^2$-conforming finite elements is difficult in general,
since they usually involve a large number of degrees of freedom \cite{ciarlet}.
Alternative solution is the application of mixed formulations or the use of
non-conforming or discontinuous Galerkin methods. In this article, we will develop
a $C^1$ conforming approximation on polygonal elements based on the Virtual Element Method (VEM).

The VEM introduced in \cite{BBCMMR2013} as a generalization of FEM which is characterized
by the capability of dealing with very general polygonal/polyhedral meshes,
including hanging nodes and nonconvex elements
(see \cite{AMSLP2018,ABMV2014,BLM2015,BMV2019,CG2017,CMS2016,FS-M2AN18,GV:IMA2017,LMRV-JSC2021,MPP2018,MR2018,PPR15} and refereneces therein).
The VEM also permits to easily implement highly regular conforming discrete spaces \cite{BM13,ChM-camwa}
which make the method very feasible to solve various fourth-order problems \cite{ABSV2016,MS-CAMWA2021,BMR2016,MRV,MV19}.
Regarding VEM for time dependent problems, we mention the following works
\cite{AN-ACOM2020,AMNS-Submitted2020,AN-NMPDE2019,ABMS2019,AMMMV-IJNME2021,VB15,V-hyperb16,ZYF-JCAM2019}.

The aim of the present paper is to introduce and analyze a virtual element method
to approximate the transverse displacement of the time dependent nonlinear
plate model problem~\eqref{modl_prob:1}--\eqref{modl_prob:5}.
The well-posedness of the continuous formulation has been studied in \cite{BGMsD-SIAP2019}.
Thus, we introduce conforming $C^1$-discretization of the problem based on the VEM for the space variable.
As demanded by our analysis, we have extended the $C^1$-VEM space introduced in \cite{BM13},
in order to compute some $L^2$ projection operators to discretize the
time dependent terms and the nonlocal term.
This newly introduced technique is capable
of handing very general polygonal meshes avoiding complex integration over elements.
Moreover, we have written a fully-discrete formulation by using a fully-implicit scheme.
We prove that the numerical solution converges to analytical
solution by using a fixed-point strategy and under standard assumptions on the computational domain,
we establish error estimates in $H^2$-norm.
Further, the appearance of nonlocal term diminish the sparse
structure of the Jacobian of the fully-discrete scheme. To avoid this difficulty, we have
introduced a new variable and retrieved the sparse structure of the Jacobian.
Further, we have proposed a linearised scheme to reduce the computational
cost without compromising the rate of convergence.  
In summary, the advantages of the proposed method are the possibility
to use general polygonal meshes with a rather straightforward construction
due to the flexibility of the virtual approach.
Moreover, the method provides an attractive and competitive
alternative in terms of its computational cost.
Finally, we mention that the method can be used to
solve the linear Kirchhoff-Love dynamic plate problem \cite{BDJ-NMPDE2005}.

The outline of this article is presented as follows. In Section~\ref{preliminaries},
the continuous weak formulation of the physical model problem  is presented.
Basic setting of functional analysis and the well-posedness of the weak formulation are
highlighted in the same section. Next, we discuss the $C^1$-VEM space and the
computation of the projection operators in Section~\ref{discrete:VEM}.
The well-posedness of the semi-discrete and fully-discrete schemes are proved in the same section.
In Section~\ref{convergence:VEM}, we discuss the convergence analysis of semi-discrete and fully-discrete schemes.
The theoretical convergence rate are justified by investigating numerical tests in Section~\ref{Numerics:VEM}. 

The major contributions of this article are enlisted as follows.
\begin{itemize}
\item The model problem deals with time dependent biharmonic term along with nonlocal
nonlinearity which is very expensive to approximate using standard finite element method.
In this article, we have  proposed an efficient, attractive and competitive virtual element
scheme for the model problem. For instance, in the lowest order case $(k = 2)$,
the total cost of the scheme is almost $3N_{\Xi}$, where $N_{\Xi}$ denotes the number
of vertices in the polygonal mesh. 
Moreover, the nonlinear scheme is solved by introducing new variable which maintains the sparsity of the jacobian.

\item The well-posedness of fully discrete scheme is proved based on some practicable
assumptions on the nonlocal coefficients $P$ and $S$ and extended the analysis for semi-discrete and fully-discrete case.
\item Using Schauder's fixed point theorem, we have derived that the solution
of fully discrete scheme belongs to a ball $\mathcal{B}_d$ with radius $d$ which
in independent of $1/\Delta t$ and depends on initial data. Therefore, we deduce
that the fully discrete scheme is stable when $\Delta t $ goes to $0$ and mesh size $h$ goes to $0$.
\item Unlike bilinear term, we have discretized the nonlocal term avoiding non-polynomial
part or stabilization part and theoretically prove that the numerical solution approximates exact solution optimally.

\end{itemize}

\section{Preliminaries and weak formulation of the problem}
\label{preliminaries}

Throughout the paper, we have dealt with the following notations. $\Omega \subset \mathbb{R}^2$
is a bounded polygonal domain with Lipschitz boundary $\Gamma:=\partial \Omega$,
corresponding to the mean surface of a plate in its reference configuration.
We assume that $\Gamma$ admits a disjoint partition $\Gamma = \Gamma_{s}\cup\Gamma_f$,
the plate being simply-supported on $\Gamma_{s}$ and free on $\Gamma_f$.
For the sake of simplicity, we also assume that both $\Gamma_{s}$ and $\Gamma_f$ have positive measure.
For the time variable, we consider $t \in \mathcal{I}:=(0,T]$, where $T$ is fixed final time.
$L^2(\Omega)$ denotes the Sobolev space of square integrable function with the norm
$\|\phi\|_{0,\O}^2:=\int_{\Omega} \phi^2 {\rm d} \O$ and the norm is induced by the inner-product
$(\phi,\psi)_{0,\O}:=\int_{\O} \phi \psi  {\rm d} \O$. The space $H^s(\O)$ consists
of functions which have square integrable derivative $D^{\alpha} \phi$ upto order $s$ and the norm associated with the space is defined as
$\|\phi\|^2_{s,\O}:=\Big( \sum_{0\leq \alpha \leq s}\|D^{\alpha}\phi\|_{0,\O}^2\Big)$,
where $\alpha$ is multi-index. Further, the space $L^2(0,T;H^s(\Omega))$ consists of function $\phi$
such that $\phi(\cdot,t)\in H^s(\Omega)$ for all most all $t \in (0,T]$ and the associated norm is defined as
$\|\phi\|^2_{L^2(0,T;H^s(\O))}:=\Big( \int_0^T \|\phi(t)\|_{s,\O}^2 {\rm d}t\Big)$.
In parallel way, we define $L^{\infty}(0,T;H^s(\O)):=\{\phi(\cdot,t) \in H^s(\O) \  \text{for all most all}\  t\in (0,T]  \}$.
Further, we define the function space 
\begin{equation*}
H^2_{\ast}(\Omega):=\{\phi \in H^2(\Omega):\phi=0 \quad \text{on} \quad  \Gamma_s \},
\end{equation*} 
with its dual space $H^2_{\ast}(\Omega)'$. Moreover, we use the angle
bracket $\langle \cdot,\cdot \rangle$ to denote the duality of $H^2_{\ast}(\Omega)'\times H^2_{\ast}(\Omega)$.
On $H^2_{\ast}(\Omega)$, we define the inner product  
\begin{equation}
(u,v)_{H^2_{\ast}(\Omega)}:= \int_{\Omega} \Big( \Delta u \Delta v-(1-\sigma)(u_{xx}v_{yy}+u_{yy}v_{xx}-2u_{xy}v_{xy})\Big) {\rm d} \O, \quad \forall u,v \in H^2_{\ast}(\Omega),
\label{inner:H2st}
\end{equation}
where $\sigma \in (0,1)$. We have that $H^2_{\ast}(\Omega)$ is a  Hilbert space (see \cite{CCCHSV-jFI2020}).

Next, we derive the weak formulation for \eqref{modl_prob:1}--\eqref{modl_prob:5}.
\subsection{Weak formulation} Let $g \in C^0([0,T],L^2(\Omega))$ for some $T>0$
and define $\mathcal{A}(u,v):=(u,v)_{H^2_{\ast}(\Omega)}$ for all $u,v \in H^2_{\ast}(\Omega)$ (cf. \eqref{inner:H2st}).
A weak solution of  \eqref{modl_prob:1}--\eqref{modl_prob:5} is a function
$u \in C^0([0,T],H^2_{\ast}(\Omega))\cap C^1 ([0,T],L^2(\Omega))\cap C^2([0,T],(H^2_{\ast}(\Omega))')$ such that
\begin{equation}
\begin{split}
& \langle D_{tt}u,v \rangle+\delta (D_t u,v)_{0,\Omega}+\mathcal{A}(u,v)+[S\|D_xu\|^2_{0,\O}
-P]a^x(u, v)=(g,v)_{0,\Omega} \\
& u(x,y,0)=u_0(x,y), \quad D_tu(x,y,0)=\omega_0(x,y),
\end{split}
\label{weak_sol}
\end{equation}
for all $t \in [0,T]$ and all $v \in H^2_{\ast}(\Omega)$, and $a^x(u,v):=(D_x u,D_x v)_{0,\Omega}$.

In order to state the well-posedness of \eqref{weak_sol}, we consider the following
eigenvalue problem: Find $(\lambda,w)\in\mathbb{R}\times H^2_{\ast}(\Omega)$, $w\ne0$, such that
\begin{equation}
\label{vibra}
\left\{\begin{array}{ll}
\Delta^2w=\lambda w& \text{in}\ \Omega,
\\
w=D_{xx}w=0 & \text{on}\ \Gamma_{s},\\
D_{yy}w+\sigma D_{xx}w=D_{yyy}w+(2-\sigma)D_{xxy}w=0 & \text{on}\ \Gamma_{f}.
\end{array}\right.
\end{equation}
We will denote by $\lambda_1>0$ the lowest eigenvalue of problem~\eqref{vibra}.

The following inequalities are going to be useful in the next sections (cf. \cite{BGMsD-SIAP2019}).
\begin{equation}\label{euthdgfk}
\Vert v\Vert_{0,\Omega}\le \Vert D_x v\Vert_{0,\Omega},\qquad
\lambda_1\Vert v\Vert_{0,\Omega}^2\le \Vert v\Vert_{2,\Omega}^2,\qquad
\lambda_1\Vert D_x v\Vert_{0,\Omega}^2\le \Vert v\Vert_{2,\Omega}^2
\qquad \forall v\in H^2_{\ast}(\O).
\end{equation}

Following \cite[Theorem~5]{BGMsD-SIAP2019}, we state the well-posedness of \eqref{weak_sol}.
\begin{theorem}\label{hgfvcxds}
Let us assume that $\delta>0$, $S>0$, $P \in [0,\lambda_1)$, $g \in C^0([0,T],L^2(\O))$,
$u_0 \in H^2_{\ast}(\Omega)$, and $\omega_0 \in L^2(\O)$, then there exists a unique weak
solution $u$ of \eqref{weak_sol}. Further, if $g\in C^1([0,T],L^2(\O))$, $u_0 \in H^4(\O) \cap H^2_{\ast}(\O)$,
and $\omega_0 \in H^2_{\ast}(\O)$, then 
\begin{equation*}
u \in C^0([0,T],H^4(\O)\cap H^2_{\ast}(\O))\cap C^1([0,T],H^2_{\ast}(\O))\cap C^2([0,T],L^2(\O))
\end{equation*}
and $u$ is a strong solution of \eqref{modl_prob:1}--\eqref{modl_prob:5}.
\end{theorem}

\section{Virtual element methods}
\label{discrete:VEM}

In this section, we describe the modified $C^1$ virtual
element space to approximate the deformation of a plate
modelled by \eqref{modl_prob:1}--\eqref{modl_prob:5}. The construction of
modified VEM space consists of several steps.
We start with the mesh construction and the assumptions
considered to introduce the discrete virtual element spaces.

Let $\{ \Omega_h \}_{h>0}$ represents sequence of decomposition
of $\Omega$ into general possibly concave polygonal element $E$
with $diam (E):=h_E$, and $diam(E):=\underset{x,y \in E}{\max}~ d_{\mathbb{R}^2}(x,y)$.
We define the mesh size $h:=\underset{E \in \O_h}{\max}~h_E$.
For all polygonal elements $E \in \O_h$, $\P_k(E)$ denotes polynomial space of degree $k$ on $E$ and 
$\mathcal{M}_{\alpha}^{\bar{d}}(E):=\Big \{ \Big(\frac{\bold{x}-\bold{x}_E}{h_E} \Big )^{s},|s|\leq \alpha  \Big \}$,
$\bar{d}=1,2$, denotes the scaled monomials, where $\bold{x}_E$ signifies the centroid of the polygon $E$.

For a particular element $E \in \O_h$, we denote by $e$ and $N_E$
the straight edges of the mesh $\O_h$ and the number of vertices of $E$, respectively,
and $\bold{n}^e_E$ denotes the unit outward normal vector to $e$ acting outward to $E$.

To analyze the discrete scheme mathematically, we will assume that
$\O_h$ satisfies the following regularity condition:
\begin{assumption}(Mesh-Regularity)
\begin{itemize}
\item Every element $E\in \O_h$ is star shaped with respect to a ball of
radius greater $\gamma h_E$, where $\gamma$ is a positive constant.
\item For every element $E$, and for all $e \subset \partial E$, $|e| \geq \gamma h_E$.
\end{itemize}
\label{mesh:regularity}
\end{assumption}

In order to introduce the discretization, for every integer $k\ge2$ and for every polygon $E$,
we define the following finite dimensional space:

\begin{align*}
\ZtK
:=\left\{\phi_h\in H^2(E) : \Delta^2\phi_h|_{E}\in\P_{k}(E), \phi_h|_{\partial E}\in C^0(\partial E),
\phi_h|_e\in\P_{r}(e)\,\,\forall e\in\partial E,\right.\\
\left.\nabla \phi_h|_{\partial E}\in C^0(\partial E)^2,
\partial_{\boldsymbol{n}^e_{ E}} \phi_h|_e\in\P_{k-1}(e)\,\,\forall e\in\partial E\right\},
\end{align*}
where $r:=\max\{3,k\}$.
Next, we identify a set of linear operators by $\boldsymbol{\chi}$ from $\ZtK$ to $\mathbb{R}$.

\paragraph{Linear operators $\boldsymbol{\chi}$.}
\begin{itemize}
\item $D1:$ The values of $\phi_h(\Xi)$ for all vertex $\Xi$;
\item $D2:$ The values of $h_\Xi\nabla \phi_h(\Xi)$ for all vertex $\Xi$;
\item $D3:$ For $r\ge4$, 
the moments 
\begin{equation*}
\frac{1}{h_e} \int_e q(\xi) \phi_h(\xi) {\rm d} \xi  \quad\forall q \in \mathcal{M}_{r-4}^1(e)
\quad \forall \text{edge}\ e;
\end{equation*}
\item $D4:$ For $k\ge3$,
the moments
\begin{equation*}
\int_e q(\xi) \partial_{\bold{n}} \phi_h (\xi)~ {\rm d} \xi
\quad \forall q \in \mathcal{M}^1_{k-3}(e) \quad\forall \text{edge}\ e;
     \end{equation*}
     \item $D5:$ For $k\ge4$, the moments 
     \begin{equation*}
         \frac{1}{h_E^2} \int_{E} q(x) \phi_h(x) {\rm d}x \quad \forall q \in \mathcal{M}_{k-4}^2(E),
     \end{equation*}
\end{itemize} 
where $h_\Xi$ corresponds to the average of the diameters
corresponding to the elements with $\Xi$ as a vertex.

To construct the modified VEM space, we introduce projection operators $\PiK$, $ \Pi^k_E$
that will be used to discretize the forms in the variational problem~\eqref{weak_sol}
and these operators are computable from the functional $\boldsymbol{\chi}$.

Let us define the projection operator $\PiK:\ZtK \rightarrow \P_{k}(E)$ 
as,
\begin{align*}
\calA(\PiK v,q)=\calA(v,q) \quad \forall  q \in \P_{k}(E) \\
\widehat{\PiK v}=\widehat{v} \ 
 \text{and}
  \  \widehat{\nabla \PiK v}=\widehat{\nabla  v},
\end{align*}
where $\widehat{v}:=\frac{1}{N_E} \sum_{i=1}^{N_E} v(\Xi_i)$ and $\Xi_i$ are the vertices of $E$. 

With the help of $\PiK$, our local virtual element space is defined
as follows,
\begin{equation}
    \ZK=\left\{ \phi_h \in \ZtK: \int_E m_{\alpha} \PiK \phi_h= \int_E m_{\alpha}
    \phi_h,\quad  m_{\alpha}\in\mathcal{M}_\alpha^{2,\star}(E),\,\,\alpha=k,k-1,k-2,k-3  \right\},
    \label{intended_space:VEM}
\end{equation}
where $\mathcal{M}_\alpha^{2,\star}(E),$ $\alpha=k,k-1,k-2,k-3$ are scaled monomials of degree $k,k-1,k-2,k-3$, respectively,
with the convention that $\mathcal{M}_{-1}^{2,\star}(E)=\emptyset$.

The global virtual element space is defined as  
\begin{equation*}
    \mathcal{Z}_h:= \left\{ \phi_h \in \Hdost: \phi_h|_{E} \in \ZK  \right\}.
\end{equation*}

From \eqref{intended_space:VEM}, it can be foreseen that the dimension
of the space $\ZK$ is same as original $C^1$-VEM space defined
in \cite{BM13}. The primary advantage of this space is that
we can compute the $L^2$-projection operator $\Pi^k_E$ onto $\P_{k}(E)$.

Next, we will prove that $\boldsymbol{\chi}$ forms DoFs of the virtual element space $\ZK$. 
With this aim, let the numbers of vertex functionals corresponding to
$D1$ and $D2$ be $\mathcal{N}_{V}$, the numbers of edge momentums
corresponding to $D3$ and $D4$ be $\mathcal{N}_{e}$ and the numbers
of cell momentums corresponding to $D5$ be $\mathcal{N}_E$.
Hence the cardinality of
$\boldsymbol{\chi}=N^{\dof}_E:=\mathcal{N}_{V}+\mathcal{N}_{e}+\mathcal{N}_{E}$.
Globally, the total DoFs will be denoted by $N^{\dof}$.

\begin{lemma}
The dimension of the auxiliary space $\ZtK$ is $\mathcal{N}_V+\mathcal{N}_e+\frac{(k+1)(k+2)}{2}$.
Moreover, the set of functionals $D1$ to $D4$ with cell moments upto order $k$ form a set of DoF for $\ZtK$.
\end{lemma}

\begin{proof}
For each element $v \in \ZtK$, we can choose the DoFs of $v$ as the trace
of $v$ on $\partial E$ and $\nabla v $ on $\partial E$ (polynomial of degree $k$)
and the moments of $v$ in $E$ upto order $k$.
Clearly, the DoFs $D1$ to $D2$ ensure $C^1$ continuity at the vertices. Further, with the help
of $D1$ and $D2$, we can identify a polynomial of degree $\leq 3$ uniquely. To compute a polynomial
of degree $r$, additional $(r-3)$ information can be obtained from $D3$.
The DoFs $D2$ and $D4$ can identify a polynomial of degree $k-1$ on each edge.
Further, proceeding in the analogous way as \cite{AABMR13}, we can prove that a
function in $H_0^2(E)$ with $\Delta^2 v \in \mathbb{P}_r(E)$, there is a mapping
between the moments of the functions upto order $r$ and their
bilaplacian of order $r$. Hence, the dimension of $\ZtK$ is
$\mathcal{N}_V+\mathcal{N}_e+\text{dim}(\P_k(E))$. Note that, the dimension
of the set containing momentum upto order $k$ is same as dimension of $\P_k(E)=\frac{(k+1)(k+2)}{2}$.
\end{proof}

In the next result, we will show that the dimension of $\mathcal{Z}_h^{E}$
is same as the dimension of $C^1$ conforming VEM space defined in \cite{BM13}.
\begin{lemma}
The dimension of $\ZK$ is $\mathcal{N}_V+\mathcal{N}_e+\frac{(k-3)(k-2)}{2}$.
The set of functionals $D1-D5$ (cf.~$\boldsymbol{\chi}$) form DoFs of the space $\ZK$.  
\end{lemma}
\begin{proof}
First, it can be observed that the dimension of
$\mathcal{M}_k^{2,\star}(E) \cup \mathcal{M}_{k-1}^{2,\star}(E) \cup \mathcal{M}_{k-2}^{2,\star}(E) \cup \mathcal{M}_{k-3}^{2,\star}(E)$
is $4k-2$. Therefore, 
\begin{equation*}
\begin{split}
\text{dim}(\ZK) &\geq \text{dim}(\ZtK)-(4k-2) \\
&=\mathcal{N}_V+\mathcal{N}_e+\frac{(k+1)(k+2)}{2}-(4k-2) \\
&=\mathcal{N}_V+\mathcal{N}_e+\frac{(k-3)(k-2)}{2}.
\end{split}                 
\end{equation*} 
Now, we prove that a function $v \in \ZK$ that vanishes on $\partial E$
with $\nabla v$ vanishes on $\partial E$ and has zero moments upto order
$k-4$ is identically zero. Clearly, $D1-D5$ are zero implies
$\PiK$ is zero which implies that all momentums of order $k, k-1, k-2$ and $k-3$ are zero.
Therefore, we deduce that $v=0$ and we conclude the proof.
\end{proof}

Weak formulation \eqref{weak_sol} consists of  non-stationary parts that
require $L^2$ projection operator to be computed.
The $C^1$ space  introduced in \cite{BM13}, does not provide enough information
to compute the orthogonal $L^2$ projection operator.
With the new space $\ZK$ we can compute $L^2$ projection operators
onto $\P_{k-1}(E)$ and $\P_{k}(E)$ keeping same the computational cost.

On a polygon $E$, we define orthogonal $L^2$ projection
operator $\Pi^{k-1}_E:\ZK \rightarrow \P_{k-1}(E)$ by 
\begin{equation*}
\int_E \Pi^{k-1}_E v_h q=\int_E v_h q \quad  \forall \ q \in \P_{k-1}(E).
\end{equation*}
Globally, the projection operator $\Pi^{k-1}$ is defined in $L^2(\Omega)$ as
\begin{equation*}
\Pi^{k-1}v_h|_E:=\Pi^{k-1}_Ev_h \quad \forall v_h \in \mathcal{Z}_h.
\end{equation*}
 
The following result shows that  $\Pi^{k-1}_E$ is computable.
\begin{lemma}
\label{projection_k-1}
The polynomial $\Pi^{k-1}_E v_h$ can be expressed explicitly
in terms of DoFs $D1-D5$ for all $v_h \in \mathcal{Z}_h^E$. 
\end{lemma}

\begin{proof}
Let $q \in \P_{k-1}(E)$. We can split $q$ into two polynomials $q_1$ and $q_2$
such that $q_1 \in \mathcal{P}_{k-1}(E)\setminus \mathcal{P}_{k-4}(E)$ and
$q_2 \in \P_{k-4}(E)$.

Using the definition of the $\Pi^{k-1}_E$, and the modified
virtual element space $\mathcal{Z}_h^E$, we have for all $q \in \P_{k-1}(E)$
\begin{equation*}
\begin{split}
\int_E \Pi^{k-1}_E v_h q&=\int_E v_h q \\
                   &=\int_E \PiK v_h q_1+\int_E v_h q_2.
\end{split}
\end{equation*}
Using cell momentum, we can compute the above integration.
\end{proof}

Now, we state that the $L^2$-orthogonal projection operator onto $\P_{k}(E)$ is also computable.
\begin{lemma}
The operator $\Pi^k_E:\mathcal{Z}_h^E \rightarrow \P_{k}(E)$ is computable for all $v_h \in \mathcal{Z}_h^E$.
\end{lemma}

\begin{proof}
Proceeding in an analogous way as in Lemma~\ref{projection_k-1}, the result follows.
\end{proof}

\subsection{Discretization of bilinear forms}\label{bili-formdfg}

In this subsection, we will employ the projection operators
to discretize the bilinear forms. For, each polygonal element $E$,
we define $\mathcal{A}_h^E(\cdot,\cdot): \ZK \times \ZK \rightarrow \mathbb{R}$, and 
$a_h^{x,E}(\cdot,\cdot): \ZK \times \ZK \rightarrow \mathbb{R}$, and $m_h^E(\cdot,\cdot):\ZK \times \ZK \rightarrow \mathbb{R}$ as follows

\begin{equation*}
\begin{split}
\mathcal{A}_h^E(\omega,v)& :=\mathcal{A}^E(\PiK \omega, \PiK v)+\mathcal{S}_{\Delta}^E((I-\PiK)\omega,(I-\PiK)v) \qquad \qquad \qquad \; \; \forall \omega,v \in \ZK,\\
a_h^{x,E}(\omega,v)&:=(\Pi^{k-1}_E D_x \omega, \Pi^{k-1}_E D_x v)_{0,E} \qquad \qquad \qquad \qquad \qquad \qquad \qquad \qquad\quad \quad \  \forall \omega,v \in \ZK,\\
m_h^E(\omega,v)&:=(\Pi^k_E \omega,\Pi^k_E v)_{0,E}+\mathcal{S}_m^E((I-\Pi^{k}_E)\omega,(I-\Pi^{k}_E)v) \qquad \qquad \qquad \qquad \qquad \forall \omega,v \in \ZK.
\end{split}
\end{equation*}

The non-polynomial parts $\mathcal{S}_{\Delta}^E(\cdot,\cdot)$
and $\mathcal{S}_{m}^E(\cdot,\cdot)$ are symmetric positive definite bilinear
forms ensure stability of the discrete forms $\mathcal{A}_h^E(\cdot,\cdot)$,
and $m_h^{E}(\cdot,\cdot)$, respectively.
Moreover, the bilinear forms satisfies the following conditions
\begin{equation*}
\begin{split}
&\alpha_{\ast} \mathcal{A}^E(v_h,v_h)\leq \mathcal{S}_{\Delta}^E(v_h,v_h) \leq \alpha^{\ast} \mathcal{A}^E(v_h,v_h) \qquad  \qquad \; \; \forall v_h \in Ker(\PiK), \\
&\gamma_{\ast} (v_h,v_h)_{0,E}\leq \mathcal{S}_{m}^E(v_h,v_h) \leq \gamma^{\ast} (v_h,v_h)_{0,E} \qquad \qquad \quad  \forall v_h \in Ker(\Pi^k_E),
\end{split}
\end{equation*} 
where $\alpha_{\ast},\alpha^{\ast},\gamma_{\ast},\gamma^{\ast}$ are positive constants independent of polygon $E$, and $Ker(\Pi)$ denotes kernel of a operator $\Pi$.
The global bilinear forms are defined as addition of local contribution.
\begin{align}
\mathcal{A}_h(\omega,v)&:=\sum_{E \in \O_h}\mathcal{A}_h^E(\omega,v) \quad \forall \omega,v \in \mathcal{Z}_h, \label{glob_biharmonic}\\
 a_h^x(\omega,v)&:=\sum_{E \in \O_h} a_h^{x,E}(\omega,v)\quad \forall \omega,v \in \mathcal{Z}_h,\label{glob_elliptic} \\
 m_h(\omega,v)&:=\sum_{E \in \O_h} m_h^E(\omega,v) \quad \forall \omega,v \in \mathcal{Z}_h.\label{glob_mass}
\end{align}

The discrete bilinear forms satisfy  consistency on polynomials and stability in the following sense.
 
\begin{lemma}
\textbf{(Polynomial Consistency)} For each polygonal element
$E\in \O_h$, and assume that $q \in \P_k(E)$ with $k\geq 2$ and $v_h \in \ZK$,
the local bilinear forms satisfy the following property:
\begin{align*}
&\mathcal{A}_h^E(q,v_h)=\mathcal{A}^E(q,v_h),  \\
&m_h^E(q,v_h)=(q,v_h)_{0,E}, \\
& a_h^{x,E}(q,v_h)=a^{x,E}(q,v_h),
\end{align*}
where $\mathcal{A}^E(\cdot,\cdot)$ and $a^{x,E}(\cdot,\cdot)$ are
restriction to element $E$ of the corresponding global forms.
\label{poly:consistency}
\end{lemma} 

\begin{lemma} 
\textbf{(Stability)} There exist four  positive constants $\aL,\aU,\gmL,\gmU$
independent of the polygon $E$ such that
\begin{align*}
& \aL\mathcal{A}^E(v_h,v_h)\leq \mathcal{A}^E_h(v_h,v_h) \leq \aU \mathcal{A}^E(v_h,v_h) \qquad \forall v_h \in \ZK,\\
& \gmL (v_h,v_h)_{0,E} \leq m_h^E(v_h,v_h) \leq \gmU (v_h,v_h)_{0,E} \qquad \forall v_h \in \ZK.
\end{align*}
\label{stability}
\end{lemma}

Further, we would like to assert that for each polygon $E \in \Omega_h$,
the discrete bilinear form $a_h^{x,E}(\cdot,\cdot)$ is bounded.
In fact, for $\omega_h,v_h \in \ZK$, we have that
\begin{equation}
\begin{split}
|a_h^{x,E}(\omega_h,v_h)| &\leq \|\Pi^{k-1}_E D_x \omega_h\|_{0,E}\|\Pi^{k-1}_E D_x v_h\|_{0,E} \\
& \leq \|D_x \omega_h \|_{0,E}\|D_x v_h\|_{0,E} \\
& \leq |\omega_h|_{1,E}|v_h|_{1,E}.
\label{continuous:ax}
\end{split}
\end{equation}

\paragraph{Discretization of nonlocal term and load term.} The model problem \eqref{modl_prob:1}
consists on geometric nonlinearity $S\int_{\O}(D_xu)^2$ which is caused due to the %because of 
stretching
of the plate in $x$ direction. By using $L^2$ projection
operator $\Pi^{k-1}_E$, we discretize the nonlocal term as
$\sum_{E \in \O_h}\int_{E}\left( \Pi^{k-1}_E D_x u_h\right)^2$.
It can be observed that for each element $E$, $ \Pi^{k-1} D_x u_h|_{E} $
is computable from $D1-D5$.  Further, we discretize the load term as follows
\begin{equation}
\begin{split}
(g_h,v_h)_h:=\sum_{E \in \O_h}(\Pi^{k}_E g, v_h)_{0,E} =\sum_{E \in \O_h}( g,\Pi^{k}_E v_h)_{0,E} \qquad \forall v_h \in \mathcal{Z}_h.\\
\end{split}
\label{load:discretization}
\end{equation} 

\subsection{Semi-discrete scheme}
\label{semi:discrete:sec}

By exploiting \eqref{glob_biharmonic}-\eqref{glob_mass} and \eqref{load:discretization},
we define the semi-discrete virtual element approximation of \eqref{weak_sol} as follows.
Find $u_h \in L^2(0,T;\mathcal{Z}_h)$ with
$D_t u_h \in L^2(0,T;\mathcal{Z}_h) $ and $D_{tt} u_h \in L^2(0,T;\mathcal{Z}_h)$ such that 
\begin{align}
& m_h(D_{tt}u_h,v_h)+\delta m_h(D_{t}u_h,v_h)+\calA_h(u_h,v_h)
+\Big[S\sum_{E \in \O_h}\int_{E} |\Pi^{k-1}_E D_x u_h|^2 -P \Big]a_h^x(u_h,v_h)\nonumber\\
&\hspace{8cm}=(g_h,v_h)_h \quad \forall v_h \in \mathcal{Z}_h \quad \text{for} \ a.e.\  t \in (0,T] .
    \label{semi_dis_schm:1}\\
    & u_h(x,y,0)=u_{h,0}, \label{semi_dis_schm:2} \\
    & D_tu_h(x,y,0)=\omega_{h,0}, \label{semi_dis_schm:3} 
\end{align}
where $u_{h,0}$ and $\omega_{h,0}$ are certain approximations
of $u_0$ and $\omega_0$, respectively.
Let us denote the nonlinear term
in \eqref{semi_dis_schm:1}
by
$$\mathcal{C}_h(u_h):=\Big[S\sum_{E \in \O_h}\Vert\Pi^{k-1}_E D_xu_h\Vert_{0,E}^2 -P \Big].$$
In what follows, we will show that the term $\mathcal{C}_h(u_h)a_h^x(u_h,v_h)$
is Lipschitz continuous. In fact, let $u_{h,1},u_{h,2}\in\mathcal{Z}_h$ be two elements.
Then, it follows that
\begin{equation}
\begin{split}
&|\mathcal{C}_h(u_{h,1})a_h^x(u_{h,1},v_h)-\mathcal{C}_h(u_{h,2})a_h^x(u_{h,2},v_h)|\\
&=\Big[S \sum_{E \in \O_h}\Vert\Pi^{k-1}_E D_xu_{h,1}\Vert_{0,E}^2 -P \Big] a_h^x(u_{h,1},v_h)
-\Big[S \sum_{E \in \O_h}\Vert\Pi^{k-1}_E D_xu_{h,2}\Vert_{0,E}^2 -P \Big] a_h^x(u_{h,2},v_h)\\
&= \Big \vert \Big[S \sum_{E \in \O_h}\Vert\Pi^{k-1}_E D_xu_{h,1}\Vert_{0,E}^2 a_h^x(u_{h,1},v_h)-S\sum_{E \in \O_h}\Vert\Pi^{k-1}_E D_xu_{h,2}\Vert_{0,E}^2 a_h^x(u_{h,2},v_h) \Big]\Big \vert\\
& \quad +P \Big \vert\Big[a_h^x(u_{h,2},v_h)-a_h^x(u_{h,1},v_h) \Big] \Big \vert\\
& =S \Big \vert\sum_{E \in \O_h}\Vert\Pi^{k-1}_E D_xu_{h,1}\Vert_{0,E}^2 a_h^x(u_{h,1}-u_{h,2},v_h) \Big \vert+ \Big \vert\Big[ S \sum_{E \in \O_h}\Vert\Pi^{k-1}_E D_xu_{h,1}\Vert_{0,E}^2\\
&\quad-S \sum_{E \in \O_h}\Vert\Pi^{k-1}_E D_xu_{h,2}\Vert_{0,E}^2 \Big] a_h^x(u_{h,2},v_h)\Big \vert+P \Big \vert a_h^x(u_{h,2}-u_{h,1},v_h)\Big \vert \\
&\leq S \|D_xu_{h,1}\|_{0,\Omega}^2 \sum_{E \in \O_h} \|D_x(u_{h,1}-u_{h,2})\|_{0,E}~\|D_xv_h\|_{0,E}+ S \sum_{E \in \O_h} \int_E \Big(|\Pi^{k-1}_E D_x u_{h,1}|\\
 & \quad  +|\Pi^{k-1}_E D_x u_{h,2}|\Big) \times\Big( |\Pi^{k-1}_E D_x u_{h,1}|-|\Pi^{k-1}_E D_x u_{h,2}|\Big) \sum_{E \in \O_h} \|D_xu_{h,2}\|_{0,E} \ \|D_xv_h\|_{0,E}\\
& \quad +P \sum_{E \in \O_h} \|D_x(u_{h,2}-u_{h,1})\|_{0,E} \|D_x v_h\|_{0,E} \\
& \leq C(S,P,\|D_xu_{h,1}\|_{0,\Omega},\|D_x u_{h,2}\|_{0,\Omega}) ~\|D_x(u_{h,1}-u_{h,2})\|_{0,\Omega}~\|D_x v_h\|_{0,\Omega}.
\end{split}
\label{Lipschitz:nonlocal}
\end{equation}

Let us assume that the matrix representation of the bilinear
forms ~$\mathcal{A}_h(\cdot,\cdot)$,~ $m_h(\cdot,\cdot)$, and $a_h^{x}(\cdot,\cdot)$
be $\bold{A},\bold{M}, \text{and}~ \bold{A}^x$, respectively.
Further, the matrix $\bold{M}$ is symmetric and positive definite,
hence the matrix is invertible. Therefore, \eqref{semi_dis_schm:1}
reduces to a system of nonlinear differential equations as below
\begin{align}
&\bold{M} \frac{d^2 \Eta_h}{dt^2}+\bold{M} \frac{d \Eta_h}{dt}+\bold{A}\Eta_h+\bold{C}(\Eta_h)=\bold{G} \label{semi:dis:non:1} \\
&\Eta_h(0)=\Eta_{h,0} \label{semi:dis:non:2}\\
&\frac{d \Eta_h}{dt}(0)=\boldsymbol{\omega}_{h,0}, \label{semi:dis:non:3}
\end{align}
where $\bold{G}$ is load vector.
Clearly \eqref{semi:dis:non:1}-\eqref{semi:dis:non:3} is a system of nonlinear
differentiable equations and the nonlinear function $\bold{C}(\Eta_h)$ is a
Lipschitz continuous with respect to $\Eta_h$. Consequently, by Picard's
Theorem on existence and uniqueness of system of differential equation,
the semi-discrete scheme \eqref{semi:dis:non:1}-\eqref{semi:dis:non:3} has an unique solution.

\subsection{Fully-discrete scheme}
In this section, we discretize the time variable by fully-implicit scheme. Let $N \in \mathbb{N}$
be a positive integer and consider the time step $\Delta t=T/N$ and the time $t_n=n~\Delta t$.
Let the approximation of $u_h$ at time $t=t_n$ is defined as $U^n_h \approx u_h(\cdot,t_n)$,
where $n=0,1,\ldots,N$. By applying {\it finite difference} 
for time variable and $C^1$-VEM for space variable, the fully discrete scheme
of the model problem~\eqref{modl_prob:1} is given by: find $U^n_h\in \mathcal{Z}_h$ such that

\begin{align}
&m_h\left(\frac{U_h^{n}-2U_h^{n-1}+U_h^{n-2}}{\Delta t^2},v_h\right)
+\delta~ m_h\left(\frac{U_h^{n}-U_h^{n-2}}{2 \Delta t},v_h\right) +  \calA_h(U_h^{n},v_h) \nonumber \\
&\quad+
\left[S\sum_{E \in \O_h} \int_{E} |\Pi^{k-1}_E D_x U_h^{n}|^2-P\right]a_h^x(U_h^{n},v_h)  
=  (g_h^{n},v_h)_h,\label{full:discrete:model} \\
&U^0_h:=I_h u_0, \qquad \text{and}~ \qquad U^1_h:=\Delta tI_h \omega_0 +U^0_h. \label{full:discrete:initial} 
\end{align}

In what follows, we would like to highlight that \eqref{full:discrete:model} is  fully implicit scheme and which is unconditionally stable.
Next, we proceed to prove the well posedness
of the fully-discrete scheme. Employing Schauder's fixed point theorem,
we show that the fully discrete scheme has unique solution $U_h^n$
at each time-step $t_n$  and the solution is bounded, i.e $\|U_h^n\|_{2,\O}\leq d$,
where $d$ is a positive constant which will be defined in the subsequent theorem.

We first recollect the  Schauder's fixed point theorem which is stated as below \cite{L69}.
\begin{lemma}
(\textbf{Schauder's Fixed Point Theorem}) Let $\mathcal{K}$ be a Banach
space and $\mathcal{B} \subset \mathcal{K}$ be a compact and convex subset.
If $\mathcal{L}:\mathcal{B} \rightarrow \mathcal{B}$ continuous mapping then $\mathcal{L}$ has a fixed point.
\label{schauder}
\end{lemma}

\begin{theorem}
\label{wellposed:full:discrete}
Let us assume that the Assumption~\ref{mesh:regularity}
on mesh regularity holds and we assumed that the prestressing
constant $P \in (0, \frac{\TAstr \lambda_1}{2})$.  Then the system of
nonlinear equations \eqref{full:discrete:model}-\eqref{full:discrete:initial}
has a solution and further, we assume that $S \in (0, \frac{\TAstr \lambda_1-P}{2~d}]$, then the solution is 
 unique, where $d$ is the radius of the ball $\calB_d$ is defined as
 $$d:=\left ( C(\TAstr,\TAstR,m_0) \|u_0\|_{2,\O}
 + C(\TGstr,\TGstR) \|\omega_0\|_{2,\O} +\Delta t~C(\delta, \TGstr)\|g\|_{L^{\infty}(0,t_{n};L^2(\O))} \right )^{1/2}.$$
\end{theorem}
\begin{proof}
We first rewrite the fully-discrete scheme~\eqref{full:discrete:model} as
\begin{equation*}
\begin{split}
m_h(U^{n}_h,v_h)+& \frac{\delta~\Delta t}{2}~m_h(U^{n}_h,v_h)+ \Delta t^2\mathcal{A}_h(U^{n}_h,v_h) + \Delta t^2 \mathcal{C}_h(U^{n}_h)a_h^x(U^{n}_h,v_h)\\
&= \Delta t^2(g_h^{n},v_h)_h+2~ m_h(U_h^{n-1},v_h)+\left(\frac{\delta~\Delta t}{2}-1\right)m_h(U_h^{n-2},v_h)\qquad\forall v_h\in\mathcal{Z}_h.
\end{split}
\end{equation*}
Further, we define a mapping 
\begin{equation*}
L_h:\mathcal{Z}_h \rightarrow \mathcal{Z}_h
\end{equation*}
by $U_h^n:=L_h(q)$, where $U_h^n$ satisfies
\begin{equation}
\begin{split}
&m_h(U_h^n,v_h)+\frac{\delta~\Delta t}{2}~m_h(U_h^n,v_h)+ \Delta t^2\mathcal{A}_h(U_h^n,v_h)+ \Delta t^2 \mathcal{C}_h(q)a_h^x(U_h^n,v_h) \\
&= \Delta t^2(g_h^{n},v_h)_h+2 m_h(U_h^{n-1},v_h)+\left(\frac{\delta~\Delta t}{2}-1\right)m_h(U_h^{n-2},v_h)\qquad\forall v_h\in\mathcal{Z}_h.
\end{split}
\label{wellposed:proof:2}
\end{equation}
 
\paragraph{Well-posedness of the mapping $\boldsymbol{L_h}$:}
 To show that the mapping $L_h$ is well-posed, it is sufficient
 to prove that $p=L_h(q)$ is the solution of the variational
 problem~\eqref{wellposed:proof:2} for each $q \in \mathcal{Z}_h$.
 
 Now, for fixed values of $\Delta t$, since problem~\eqref{wellposed:proof:2} is a square linear
 system it is enough to prove uniqueness. To this end we assume that
the right hand side vanish, and we test the problem with $v_h=U_h^n$,
to obtain that
\begin{equation*}
\begin{split}
0=&\Big(1+\frac{\delta~\Delta t}{2}\Big)\ m_h(U_h^n,U_h^n)+\Delta t^2 \mathcal{A}_h(U_h^n,U_h^n)
+\Delta t^2 \mathcal{C}_h(q)a_h^x(U_h^n,U_h^n)\\
& \geq \Big(1+\frac{\delta~\Delta t}{2}\Big)~\gmL~\|U_h^n\|_{0,\O}^2+\Delta t^2~\aL~\|U_h^n\|_{2,\O}^2
+\Delta t^2~S \Big( \sum_{E \in \O_h} \|\Pi^{k-1}_E D_x q\|^2_{0,E}\Big) \\
& \quad \times\Big( \sum_{E \in \O_h} \|\Pi^{k-1}_E D_x U_h^n\|^2_{0,E}\Big)-\Delta t^2~ P~\Big( \sum_{E \in \O_h}\|\Pi^{k-1}_E D_xU_h^n\|^2_{0,E} \Big) \\
& \geq \Big(1+\frac{\delta~\Delta t}{2}\Big)~\gmL~\|U_h^n\|_{0,\Omega}^2+\Delta t^2~\aL~\lambda_1 ~\|D_xU_h^n\|_{0,\Omega}^2-\Delta t^2 P \|D_x U_h^n\|^2_{0,\Omega}\\
&\geq \Big(1+\frac{\delta~\Delta t}{2}\Big)~\gmL~\|U_h^n\|_{0,\Omega}^2+\Delta t^2 \Big(\aL~\lambda_1-P \Big)\|D_x U_h^n\|_{0,\Omega}^2\geq 0,
\end{split}
\end{equation*}
where we have used the third inequality in \eqref{euthdgfk}. This shows that $U_h^n=0$.
Thus, \eqref{wellposed:proof:2} has unique solution and $L_h$ is well-defined.

Next, we show that the mapping $L_h:\mathcal{B}_d\rightarrow \mathcal{B}_d$
maps a closed ball to a closed ball of radius $d \in \mathbb{R}$,
where $\mathcal{B}_d:=\{v_h\in \mathcal{Z}_h~:~\|v_h\|_{2,\O}\leq d \}$.

By choosing test function $v_h:= \frac{U_h^n-U_h^{n-2}}{\Delta t}$ in \eqref{wellposed:proof:2}, we obtain
\begin{equation*}
\begin{split}
& m_h \Big(\frac{U_h^n-2 U_h^{n-1}+U_h^{n-2}}{\Delta t^2},\frac{U_h^n-U_h^{n-2}}{\Delta t} \Big)
+\delta ~m_h \left ( \frac{U_h^{n}-U_h^{n-2}}{2 \Delta t},\frac{U_h^{n}-U_h^{n-2}}{\Delta t}\right )\\
&+\mathcal{A}_h \left (U_h^n,\frac{U_h^n-U_h^{n-2}}{\Delta t} \right )
+ \mathcal{C}_h(q) a_h^x \left(U_h^n,\frac{U_h^n-U_h^{n-2}}{\Delta t} \right )=\left (g_h^n,\frac{U_h^n-U_h^{n-2}}{\Delta t} \right )_h.
\end{split}
\end{equation*}
By using stability of $m_h(\cdot,\cdot)$ (cf. Lemma~\ref{stability}), we obtain
\begin{equation*}
\begin{split}
& \frac{C(\TGstr,\TGstR)}{\Delta t }
\left ( \left\|\frac{U_h^n-U^{n-1}_h}{\Delta t}\right\|_{0,\Omega}^2-\left\|\frac{U_h^{n-1}-U^{n-2}_h}{\Delta t}\right\|_{0,\Omega}^2 \right )
+ \frac{ \TGstr \delta}{2} \left\|\frac{U_h^n-U_h^{n-2}}{\Delta t}\right\|_{0,\Omega}^2 \\
& \quad +\frac{1}{\Delta t}\mathcal{A}_h(U_h^n,U_h^n) 
+ \frac{1}{\Delta t}\left (S\sum_{E\in \Omega_h} \|\Pi^{k-1}_E D_x q \|_{0,E}^2-P \right) a_h^x(U_h^n,U_h^n)\\
& \leq \frac{1}{\Delta t}|\mathcal{A}_h(U_h^n,U_h^{n-2})|+ \frac{m_0}{\Delta t} ~|a_h^x\left(U_h^n,U_h^{n-2}\right)  |
+\|g_h^n\|_{0,\Omega}~\left\|\frac{U_h^n-U_h^{n-2}}{\Delta t}\right\|_{0,\Omega}, 
\end{split}
\end{equation*}
where we have bounded $\vert\mathcal{C}_h(q)\vert\le m_0\quad\forall q\in\mathcal{Z}_h.$

Further, using the assumption on $P$ and the third inequality in \eqref{euthdgfk}, we obtain
\begin{equation}
\begin{split}
& \frac{C(\TGstr,\TGstR)}{\Delta t }
\left ( \left\|\frac{U_h^n-U^{n-1}_h}{\Delta t}\right\|_{0,\Omega}^2-\left\|\frac{U_h^{n-1}-U^{n-2}_h}{\Delta t}\right\|_{0,\Omega}^2 \right )
+ \frac{ \TGstr \delta}{2} \left\|\frac{U_h^n-U_h^{n-2}}{\Delta t}\right\|_{0,\Omega}^2 \\
& \quad +\frac{\TAstr}{2~\Delta t} \|U_h^n\|_{2,\Omega}^2+ \frac{1}{\Delta t}~\Big (\frac{\TAstr~\lambda_1}{2}-P \Big )\|D_x U_h^n\|_{0,\O}^2+\frac{1}{\Delta t}\left (S\sum_{E\in \Omega_h} \|\Pi^{k-1}_E D_x q \|_{0,E}^2 \right ) \|D_x U_h^n\|_{0,\Omega}^2 \\
& \leq \frac{\TAstR}{\Delta t} \|U_h^n\|_{2,\Omega} \|U_h^{n-2}\|_{2,\O}+\frac{m_0}{\Delta t} \|D_x U_h^n\|_{0,\O}~\|D_x U_h^{n-2}\|_{0,\O}+\|g_h^n\|_{0,\Omega}~\left\|\frac{U_h^n-U_h^{n-2}}{\Delta t}\right\|_{0,\Omega}.
\label{Ball:new}
\end{split}
\end{equation}
Multiplying by $\Delta t$ on both side of \eqref{Ball:new} and using Young's inequality,
and neglecting the term
$\frac{1}{\Delta t}\left (S\sum_{E\in \Omega_h} \|\Pi^{k-1}_E D_x q \|_{0,E}^2 \right ) \|D_x U_h^n\|_{0,\O}^2$, we derive
\begin{equation}
\begin{split}
C(\TGstr,\TGstR)~\left\|\frac{U_h^n-U^{n-1}_h}{\Delta t}\right\|_{0,\Omega}^2&+ \frac{ \TGstr \delta \Delta t}{4} \left\|\frac{U_h^n-U_h^{n-2}}{\Delta t}\right\|_{0,\Omega}^2+\frac{\TAstr}{4} \|U_h^n\|_{2,\O}^2 \\
& \quad + \Big (\frac{\TAstr~\lambda_1}{2}-P \Big )\|D_x U_h^n\|_{0,\O}^2 \leq C(\TGstr,\TGstR)~\left\|\frac{U_h^{n-1}-U^{n-2}_h}{\Delta t}\right\|_{0,\Omega}^2 \\ &\quad + C(\TAstr,\TAstR)\|U_h^{n-2}\|_{2,\O}^2  +C(m_0,\widetilde{\alpha_{\ast}},\lambda_1)\|D_x U_h^{n-2}\|_{0,\O}^2 \\
& \quad +\Delta t~C(\delta,\TGstr)\|g^n\|^2_{0,\Omega}.
\label{Ball:rearrange} 
\end{split}
\end{equation}
Neglecting the term $\frac{ \TGstr \delta \Delta t}{4} \left\|\frac{U_h^n-U_h^{n-2}}{\Delta t}\right\|_{0,\Omega}^2$
in \eqref{Ball:rearrange} and  rearranging the terms,  we obtain
\begin{equation*}
\begin{split}
\|U_h^n\|_{2,\O}^2 \leq C(\TAstr,\TAstR,m_0) \|U_h^0\|_{2,\O}^2&+ C(\TGstr,\TGstR) \left \|\frac{ U_h^1-U_h^0}{\Delta t} \right \|_{0,\O}\\
& \quad +\Delta t~C(\delta, \TGstr)\|g\|_{L^{\infty}(0,t_{n};L^2(\O))}^2.
\end{split}
\end{equation*}

We define 
\begin{equation}
d:=\left ( C(\TAstr,\TAstR,m_0) \|u_0\|_{2,\O}^2+ C(\TGstr,\TGstR) \|\omega_0\|_{2,\O}^2 +\Delta t~C(\delta, \TGstr)\|g\|_{L^{\infty}(0,t_{n};L^2(\O))}^2 \right )^{1/2}.
\label{d}
\end{equation}
Therefore, using the boundedness property of the interpolation operator $I_h$ in $\|\cdot \|_{2,\O}$ norm,
we conclude that $L_h:\mathcal{B}_d \rightarrow \mathcal{B}_d$,
where the radius $d$ is defined in \eqref{d}.

\paragraph{Continuity of $\boldsymbol{L_h}$:}
Let $\epsilon>0$ be a small number and $q,q_0 \in \mathcal{B}_d$
be two given elements such that $\|q-q_0\|_{2,\Omega}<\bar{\delta}$
where $\bar{\delta}>0$ is a fixed small number could be depend on $\epsilon$. We will show that 
\begin{equation*}
\|q-q_0\|_{2,\Omega}\leq \bar{\delta}\qquad \Longrightarrow\qquad \|U_h^n-p_0\|_{2,\Omega} < \epsilon,
\end{equation*}
 where $U_h^n=L_h(q)$ and $p_0=L_h(q_0)$. Using \eqref{wellposed:proof:2}, we obtain 
\begin{equation}
\begin{split}
&m_h(U_h^n-p_0,v_h)+\frac{\delta \Delta t}{2} m_h(U_h^n-p_0,v_h)+\Delta t^2 \mathcal{A}_h(U_h^n-p_0,v_h)+\Delta t^2 P a_h^x(p_0-U_h^n,v_h)\\
& \quad +\Delta t^2 \Big(S \sum_{E \in \O_h} \int_{E}| \Pi^{k-1}_E D_x  q|^2 a_h^x(U_h^n,v_h)-S \sum_{E \in \O_h} \int_{E}|\Pi^{k-1}_E D_x q_0|^2a_h^x(p_0,v_h) \Big) =0.
\end{split}
\label{wellposed:continuity}
\end{equation}
Adding and subtracting the term $S \sum_{E \in \O_h} \int_{E} |\Pi^{k-1}_E D_x q|^2 a_h^x(p_0,v_h)$
with \eqref{wellposed:continuity}, we obtain
\begin{equation}
\begin{split}
&m_h(U_h^n-p_0,v_h)+\frac{\delta \Delta t}{2}~m_h(U_h^n-p_0,v_h)+\Delta t^2~\mathcal{A}_h(U_h^n-p_0,v_h) 
+\Delta t^2~ \mathcal{C}_h(q)a_h^x(U_h^n-p_0,v_h) \\
& = \Delta t^2 \Big( S \sum_{E \in \O_h} \int_{E}|\Pi^{k-1}_E D_x q_0|^2-S \sum_{E \in \O_h} \int_{E}|\Pi^{k-1}_E D_x q|^2 \Big) a_h^x(p_0,v_h).
\end{split}
\label{wellposed:diff}
\end{equation}
By choosing, $v_h:=U_h^n-p_0$ in \eqref{wellposed:diff}, we obtain
\begin{equation}
\begin{split}
&m_h(U_h^n-p_0,U_h^n-p_0)+\frac{\delta \Delta t}{2}~m_h(U_h^n-p_0,U_h^n-p_0)+\Delta t^2~\mathcal{A}_h(U_h^n-p_0,U_h^n-p_0) \\
&+\Delta t^2~ \mathcal{C}_h(q)a_h^x(U_h^n-p_0,U_h^n-p_0) \\
&= \Delta t^2~\Big( S \sum_{E \in \O_h}\int_{E}|\Pi^{k-1}_E D_x q_0|^2-S \sum_{E \in \O_h}\int_{E}|\Pi^{k-1}_E D_x q|^2 \Big) a_h^x(p_0,U_h^n-p_0).
\end{split}
\label{wellposed:continuity:new}
\end{equation}

Further, using the boundedness of the projection operator $\Pi^{k-1}$
and using that $q,q_0 \in \mathcal{B}_{d}$, we rewrite the difference on the right hand side as follows
\begin{equation}
\begin{split}
&\Big( S \sum_{E \in \O_h} \int_{E}|\Pi^{k-1}_E D_x q_0|^2-S \sum_{E \in \O_h} \int_{E}|\Pi^{k-1}_E D_x q|^2 \Big)\\ 
& = S \sum_{E \in \O_h}\int_{E} (|\Pi^{k-1}_E D_x q_0|+|\Pi^{k-1}_E D_x q|)(|\Pi^{k-1}_E D_x q_0|-|\Pi^{k-1}_E D_x q|)\\
& \leq 2Sd~ \sum_{E\in \O_h}\| \Pi^{k-1}_E D_x(q_0-q)\|_{0,E} \\
& \leq 2Sd~|q_0-q|_{1,\Omega}. \qquad \qquad \qquad \qquad \qquad \qquad\text{(Recollecting  \eqref{Lipschitz:nonlocal})}
\end{split}
\label{wellposed:nonlinear:diff}
\end{equation}

By applying the stability of $m_h(\cdot,\cdot)$, $\mathcal{A}_h(\cdot,\cdot)$
(cf. Lemma~\eqref{stability}) continuity of 
 $a_h^x(\cdot,\cdot)$ (cf. \eqref{continuous:ax})
in \eqref{wellposed:continuity:new},
and  boundedness of $p_0$ such as $\|p_0\|_{2,\Omega}\leq d$,
and using the third inequality in \eqref{euthdgfk}, we obtain
\begin{equation}
\begin{split}
 \TGstr \|U_h^n-p_0\|_{0,\Omega}^2+ \frac{\delta \Delta t}{2} \TGstr \|U_h^n-p_0\|_{0,\Omega}^2&
 +\frac{\Delta t^2~\TAstr}{2}~\|U_h^n-p_0\|_{2,\Omega}^2+\Delta t^2\Big( \frac{\TAstr \lambda_1}{2}-P \Big) \|D_x(U_h^n-p_0)\|_{0,\Omega}^2 \\
 & \quad+ \Delta t^2~\Big( S \sum_{E \in \Omega_h} \|\Pi^{k-1}_E D_x u_h\|_{0,E}^2 \Big)~\|D_x(U_h^n-p_0)\|_{0,\Omega}^2 \\
 & \quad \leq 
  2Sd~ \Delta t^2 \|D_x p_0\|_{0,\Omega}~\|q_0-q\|_{2,\Omega}~\|D_x(U_h^n-p_0)\|_{0,\Omega}. 
\end{split}
\label{wellposedness:inequality}
\end{equation}

Neglecting the term $ \Big( S \sum_{E \in \Omega_h} \|\Pi^{k-1}_E D_x u_h\|_{0,E}^2 \Big)~\|D_x(U_h^n-p_0)\|_{0,\Omega}^2$
on the left-side of \eqref{wellposedness:inequality}, we obtain
\begin{equation*}
\begin{split}
\TGstr \|U_h^n-p_0\|_{0,\Omega}^2+ \frac{\delta \Delta t}{2} \TGstr \|U_h^n-p_0\|_{0,\Omega}^2&
 +\frac{\Delta t^2~\TAstr}{2}~\|U_h^n-p_0\|_{2,\Omega}^2+\Delta t^2~\Big( \frac{\TAstr \lambda_1}{2}-P \Big) \|D_x(U_h^n-p_0)\|_{0,\Omega}^2 \\
 & \quad \leq 
  2Sd~ \Delta t^2 \|D_xp_0\|_{0,\Omega}~\|q_0-q\|_{2,\Omega}~\|D_x(U_h^n-p_0)\|_{0,\Omega}. 
\end{split}
\end{equation*}

Further, using Young's inequality and kick-back arguments, we obtain
\begin{equation*}
\begin{split}
\TGstr \|U_h^n-p_0\|_{0,\Omega}^2+& \frac{\delta \Delta t}{2} \TGstr \|U_h^n-p_0\|_{0,\Omega}^2+\frac{\Delta t^2~\TAstr}{2}~\|U_h^n-p_0\|_{2,\Omega}^2 +\Delta t^2~\Big( \frac{\TAstr \lambda_1-2P}{4} \Big)\\
& \quad \times \|D_x(U_h^n-p_0)\|_{0,\Omega}^2 \leq 
  \frac{(8~S^2~ d^4)}{( \TAstr \lambda_1-2P )}~ \Delta t^2 ~\|q_0-q\|_{2,\Omega}^2.
\end{split}
\end{equation*}

Since the coefficients of $\|U_h^n-p_0\|_{0,\Omega}$ and $\|D_x(U_h^n-p_0)\|_{0,\Omega}$ are positive, hence neglecting the terms, we obtain
\begin{equation*}
\|U_h^n-p_0\|_{2,\Omega}^2 \leq \frac{(16~S^2~d^4)}{( \TAstr \lambda_1-2P )~\TAstr} \|q_0-q\|_{2,\Omega}^2,
\end{equation*} 
which implies that the mapping $L_h$ is continuous.

Hence from Schauder's fixed point theorem (cf. Lemma~\ref{schauder}),
we can conclude that $L_h$ has a fixed point, i.e, $q=L_h(q)$
which implies that there a solution of the nonlinear equation \eqref{full:discrete:model}.

\paragraph{Uniqueness of the solution:}
Let $U^1_n$ and $U^2_n$ be two numerical solutions of \eqref{full:discrete:model}.
Then from \eqref{full:discrete:model}, we obtain 
\begin{equation}
\begin{split}
&m_h(U_n^1-U_n^2,v_h)+\frac{\delta \Delta t}{2} m_h(U_n^1-U_n^2,v_h)+\Delta t^2 \mathcal{A}_h(U_n^1-U_n^2,v_h)+\Delta t^2 P a_h^x(U_n^2-U_n^1,v_h)\\
& \quad +\Delta t^2 \Big(S \sum_{E \in \O_h} \int_{E}|\Pi^{k-1}_E D_x U_n^1|^2 a_h^x(U_n^1,v_h)-S \sum_{E \in \O_h} \int_{E}|\Pi^{k-1}_E D_x U_n^2|^2a_h^x(U_n^2,v_h) \Big)=0.
\end{split}
\label{wellposed:uniqueness}
\end{equation}

By choosing $v_h:=U_n^1-U_n^2$ in \eqref{wellposed:uniqueness}, we obtain
\begin{equation*}
\begin{split}
&m_h(U_n^1-U_n^2,U_n^1-U_n^2)+\frac{\delta \Delta t}{2} m_h(U_n^1-U_n^2,U_n^1-U_n^2)+\Delta t^2 \mathcal{A}_h(U_n^1-U_n^2,U_n^1-U_n^2)\\
& \quad +\Delta t^2 \Big(S \sum_{E \in \O_h} \int_{E}|\Pi^{k-1}_E D_x  U_n^1|^2 a_h^x(U_n^1,U_n^1-U_n^2)-S \sum_{E \in \O_h} \int_{E}|\Pi^{k-1}_E D_x U_n^2|^2a_h^x(U_n^2,U_n^1-U_n^2) \Big)\\
& \quad - P \Delta t^2 a_h^x(U_n^1-U_n^2,U_n^1-U_n^2)=0.
\end{split}
\end{equation*}

Further, an application of stability property of $m_h(\cdot,\cdot)$, and $\mathcal{A}_h(\cdot,\cdot)$,
and continuity of $a_h^x(\cdot,\cdot)$ (cf. Lemma~\ref{stability}) and following
same arguments as \eqref{wellposed:nonlinear:diff}, we obtain
\begin{equation*}
\begin{split}
\TGstr \|U_n^1-U_n^2\|_{0,\Omega}^2 &+ \frac{\delta \Delta t}{2} \TGstr\|U_n^1-U_n^2\|_{0,\Omega}^2
+ \Delta t^2 \TAstr \lambda_1\|D_x(U_n^1-U_n^2)\|_{0,\Omega}^2 \\
&+\Delta t^2\Big(S~\sum_{E \in \O_h} \int_E |\Pi^{k-1}_E D_x U_n^2|^2-P \Big) \sum_{E \in \O_h} \Vert\Pi^{k-1}_E D_x(U_n^1-U_n^2)\Vert_{0,E}^2 \\
& \leq 2Sd~\Delta t^2~\| D_x(U_n^1- U_n^2)\|^2_{0,\Omega}.
\end{split}
%\label{wellposed:unique:2}
\end{equation*}

Further, using the third inequality in \eqref{euthdgfk}
and the assumptions of Theorem~\ref{wellposed:full:discrete} on $P$, we derive
\begin{equation}
\begin{split}
\TGstr \|U_n^1-U_n^2\|_{0,\Omega}^2 &+ \frac{\delta \Delta t}{2} \TGstr\|U_n^1-U_n^2\|_{0,\Omega}^2
+ \Delta t^2 \Big( \TAstr \lambda_1-P \Big)\|D_x(U_n^1-U_n^2)\|_{0,\Omega}^2 \\
&+\Delta t^2\Big(S~\sum_{E \in \O_h} \int_E |\Pi^{k-1}_E D_x U_n^2|^2 \Big) \sum_{E \in \O_h} \Vert\Pi^{k-1}_E D_x(U_n^1-U_n^2)\Vert_{0,E}^2 \\
& \leq 2Sd~\Delta t^2~\| D_x(U_n^1- U_n^2)\|^2_{0,\Omega}.
\end{split}
\label{wellposed:unique:mod}
\end{equation}

Since the term $\Big( \Delta t^2~S~\sum_{E \in \O_h} \int_E |\Pi^{k-1}_E D_x U_n^2|^2 ~ \sum_{E \in \O_h}\Vert\Pi^{k-1}_E D_x(U_n^1-U_n^2)\Vert_{0,E}^2\Big)$
is positive, neglecting from left-hand side of \eqref{wellposed:unique:mod}, we obtain
\begin{equation*}
\begin{split}
\TGstr \|U_n^1-U_n^2\|_{0,\Omega}^2 &+ \frac{\delta \Delta t}{2} \TGstr\|U_n^1-U_n^2\|_{0,\Omega}^2
+ \Delta t^2 \Big( \TAstr \lambda_1-P \Big)\|D_x(U_n^1-U_n^2)\|_{0,\Omega}^2 \\
& \leq 2Sd~\Delta t^2~\|D_x( U_n^1- U_n^2)\|^2_{0,\Omega}.
\end{split}
\end{equation*}

A straightforward mathematical evaluation implies that
\begin{equation*}
\begin{split}
\frac{\TGstr}{2~d} \|U_n^1-U_n^2\|_{0,\Omega}^2 &+ \frac{\delta \Delta t}{4~d} \TGstr\|U_n^1-U_n^2\|_{0,\Omega}^2
+ \Delta t^2 \Big( \frac{\TAstr \lambda_1-P}{2~d}-S \Big)\|D_x(U_n^1-U_n^2)\|_{0,\Omega}^2  \leq 0.
\end{split}
\end{equation*}

Further, the assumption of Theorem~\ref{wellposed:full:discrete} on $S$ implies that the
term $\Big( \frac{\TAstr \lambda_1-P}{2~d}-S \Big)>0$, and hence neglecting the term
corresponding to $\|D_x(U_n^1-U_n^2)\|_{0,\Omega}$ and $\|U_n^1-U_n^2\|_0$, we obtain
 \begin{equation*}
\|U_n^1-U_n^2\|_{0,\Omega}^2 \leq 0,
\end{equation*}
which implies $U_n^1=U_n^2$.
\end{proof}

\begin{remark}
In the proof of Theorem~\ref{wellposed:full:discrete}, we have proved
that the fully-discrete scheme \eqref{full:discrete:model}-\eqref{full:discrete:initial}
has unique solution based on some assumptions on $P$ and $S$ which are explicitly
stated as  $P \in (0, \frac{\TAstr ~\lambda_1}{2})$ and $S \in (0, \frac{\TAstr \lambda_1-P}{2d})$.
In comparison with the wellposedness of continuous weak formulation (Theorem~\ref{hgfvcxds}),
we have considered analogous assumption on $P$. The positive constant $\TAstr$
appeared due to discrete virtual element formulation. In addition, the fully-discrete
scheme posses unique solution for sufficiently small values of $S$. Further, we would
like to highlight that the proof of wellposedness of fully-discrete scheme is
independent of small values of time-step, i.e., $\Delta t$ and could be completely controlled by initial data.
\end{remark}

\subsection{Implementation of fully-discrete scheme}
\label{plate_implement}

In this section, we discuss the implementation procedure of the fully-discrete scheme.
By employing fully-implicit scheme in time variable and $C^1$-VEM in space variable,
the fully-discrete scheme \eqref{full:discrete:model} reduces to system of
nonlinear equations for each time-steps $t_n$, $0\leq n \leq N$. The nonlinear
system can be solved by employing any iteration technique such as Picard's
iteration technique or Newton's method. Since the Newton's method converges
with higher order ({\it rate of convergence is 2}) compared to Picard's
iteration technique ({\it rate of convergence is 1}), we will utilize
Newton's method to solve the nonlinear system~\eqref{full:discrete:model}.
However, the primary difficulty with Newton's methods is that the computation
of Jacobian which needs to be updated at each time-step. Further, the presence
of nonlocal term $\mathcal{C}_h(U_h^n)a_h^x(U_h^n,v_h)$
disfigures the sparse structure of the Jacobian. Consequently, the computational
cost is increased. To avoid this difficulty, we introduce a new independent
variable and maintain the sparse structure of the Jacobian. Let $\{ \phi_i \}_{1\leq i \leq N^{\dof}}$
be the global basis of $\mathcal{Z}_h$ associated with the DoFs $\boldsymbol{\chi}$ of $\O_h$.
By applying the $\delta_{ij}$ property of the basis function $\phi_i$ of $\mathcal{Z}_h$,
we rewrite the discrete solution as
\begin{equation}
U_h^n=\sum_{j=1}^{N^{\dof}} \eta^n_j~\phi_j,
\label{discrete:sol:explicit}
\end{equation}
where $\eta^n_j:=\chi_j(U_h^n)$ is the coefficient of the basis function $\phi_j$.
Collecting all coefficient $\eta^n_j$, we constitute the coefficient vector say
$\Eta^n:=[\eta^n_1,\eta^n_2,\ldots,\eta^n_{N^{\dof}} ]^{T}$.
Using \eqref{discrete:sol:explicit}, we rewrite \eqref{full:discrete:model}
into system of algebraic equation as follows.
\begin{equation*}
\mathcal{E}_i(\Eta^n)=0 \quad 1\leq i \leq N^{\dof},
\end{equation*}
where
\begin{equation}
\begin{split}
\mathcal{E}_i(\Eta^n)&=m_h(U^{n}_h,\phi_i)+\frac{\Delta t~\delta}{2}m_h(U_h^{n},\phi_i)+\Delta t^2 \mathcal{A}_h(U^{n}_h,\phi_i) \\
&+ \Delta t^2 \mathcal{C}_h(U^{n}_h)a_h^x(U^n_h,\phi_i)-\Delta t^2~(g_h^{n},v_h)_h\\
&-2m_h(U_h^{n-1},\phi_i)+m_h(U_h^{n-2},\phi_i)-\frac{\Delta t~\delta}{2} m_h(U_h^{n-2},\phi_i). 
\end{split}
\label{algebraic:nonlocal}
\end{equation}

Each entries of the Jacobian matrix $\bold{J}$ is formulated as below
\begin{equation}
\begin{split}
(\bold{J})_{ij}:=\frac{\partial \mathcal{E}_i(\Eta^n) }{\partial \eta^n_j}& =m_h(\phi_j,\phi_i)+\frac{\Delta t~\delta}{2} m_h(\phi_j,\phi_i)+\Delta t^2 \mathcal{A}_h(\phi_j,\phi_i) \\
& \quad + 2~\Delta t^2~\Big [S \sum_{E \in \O_h} \int_{E} \Pi^{k-1}_E D_x U_h^n ~\Pi^{k-1}_E D_x  \phi_j \Big ]a_h^x(U_h^n,\phi_i) \\
& \quad + \Delta t^2~\mathcal{C}_h(U^{n}_h)a_h^x(\phi_j,\phi_i).
\end{split}
\label{jacobian}
\end{equation}

It can be observed that the Jacobian $\bold{J}$ is full matrix,
hence numerically expensive to implement. To avoid this difficulty,
we exploit the idea by extending the independent variable which is
presented in \cite{G12}. By doing this, we rewrite the system as follows:
Find $(U_h^n, \xi) \in \mathcal{Z}_h \times \mathbb{R}$ such that
\begin{equation}
\begin{split}
&\mathcal{E}_i(U_h^n,\xi):=m_h(U^{n}_h,\phi_i)+\frac{\Delta t~\delta}{2}m_h(U_h^{n},\phi_i)+\Delta t^2 \mathcal{A}_h(U^{n}_h,\phi_i) \\
&+ \Delta t^2 \Big [ S \xi-P\Big]a_h^x(U^n_h,\phi_i)-\Delta t^2~(g_h^{n},v_h)_h\\
&-2m_h(U_h^{n-1},\phi_i)+m_h(U_h^{n-2},\phi_i)-\frac{\Delta t~\delta}{2} m_h(U_h^{n-2},\phi_i), \quad 1\leq i \leq N^{\dof} \\
&\mathcal{E}_{N^{\dof}+1}(U_h^n,\xi):= \sum_{E \in \O_h}\int_{E}|\Pi^{k-1}_E D_x  U_h^n|^2-\xi.
\end{split}
\label{new:nonlinear_system}
\end{equation}
The Jacobian of the system \eqref{new:nonlinear_system} is given by
\begin{equation*}
\widetilde{\bold{J}}:=
\begin{bmatrix}
& \bold{J}_1 & \bold{J}_2 \\
& \bold{J}_3 & \bold{J}_4
\end{bmatrix}_{(N^{\dof}+1)\times(N^{\dof}+1)}. 
\end{equation*}
$\textbf{J}_1$ is the jacobian corresponding to the system of linear equations
\begin{equation*}
\mathcal{E}_i(U^n_h,\xi)=0 \quad  1\leq i \leq N^{\dof}, 
\end{equation*}
which is given by
\begin{equation*}
\begin{split}
(\textbf{J}_1)_{i,j}:=\frac{\partial \mathcal{E}_i(U_h^n,\xi) }{\partial \eta^n_j}&=m_h(\phi_j,\phi_i)+\frac{\Delta t~\delta}{2} m_h(\phi_j,\phi_i)+\Delta t^2 \mathcal{A}_h(\phi_j,\phi_i) \\
& \quad +\Delta t^2 (S \xi-P) a_h^x(\phi_j,\phi_i), \quad 1\leq i,j \leq N^{\dof}.
\end{split}
\end{equation*}
$\textbf{J}_2$ and $\textbf{J}_3^{T}$ are the column vectors which are defined as follows: 
\begin{equation*}
\begin{split}
(\textbf{J}_2)_{i,1}&:=\frac{\partial \mathcal{E}_i(U_h^n,\xi) }{\partial \xi} \\
&=\Delta t^2~S~a_h^x(U_h^n,\phi_i) \quad 1\leq i \leq N^{\dof},
\end{split}
\end{equation*}
and
\begin{equation*}
\begin{split}
(\textbf{J}_3)_{1,j}&:=\frac{\partial \mathcal{E}_{N^{\dof}+1}(U_h^n,\xi) }{\partial \eta^n_j} \\
&=2 \sum_{E \in \O_h}\int_{E} \Pi^{k-1}_E D_x U_h^n~ \Pi^{k-1}_E D_x \phi_j, \quad 1\leq j \leq N^{\dof}.
\end{split}
\end{equation*}
Further, the matrix with single entry $\textbf{J}_4$ can be expressed as
\begin{equation*}
(\textbf{J}_4)_{1\times 1}:= \frac{\partial \mathcal{E}_{N^{\dof}+1}(U_h^n,\xi) }{\partial \xi}=-1.
\end{equation*}

In continuation, we would like to highlight that in the discrete formulation,
we have considered nonlinearity of the polynomial part of $U_h^n$ only
such as $\Big[ \sum_{E \in \O_h}\int_{E} | \Pi^{k-1}_E D_x U_h^n|^2 {\rm d} E \Big]$.
In the next section, we will prove that the proposed approximation ensures optimal
rate of convergence.
The following result, which proof follows standard arguments, shows
that \eqref{algebraic:nonlocal} and \eqref{new:nonlinear_system} are equivalent.
\begin{theorem}
Let us assume that Assumption~\ref{mesh:regularity} holds and $U_h^n \in \mathcal{Z}_h$
be the solution of \eqref{algebraic:nonlocal}, then the pair $(U_h^n, \xi)\in \mathcal{Z}_h \times \mathbb{R}$
be the solution of \eqref{new:nonlinear_system}. Conversely, let $(U_h^n,\xi)\in\mathcal{Z}_h \times \mathbb{R}$
satisfies \eqref{new:nonlinear_system}, then $U_h^n \in \mathcal{Z}_h$ satisfies \eqref{algebraic:nonlocal}.
\end{theorem}

\subsection{Linearized scheme}
\label{section:linearized}
In the previous subsection, we have highlighted that the presence of nonlocal
term increases the computational cost of \eqref{full:discrete:model}.
To avoid this shortcoming, we propose a linearized scheme without compromising
the rate of convergence as follows: for $n=2,\ldots,N$,
find $\widetilde{U^{n}_h}\in \mathcal{Z}_h$ such that   
\begin{align}
&m_h\left(\frac{\widetilde{U_h^{n}}
-2\widetilde{U_h^{n-1}}+\widetilde{U_h^{n-2}}}{\Delta t^2},v_h\right)+\delta~ m_h\left(\frac{\widetilde{U_h^{n}}-\widetilde{U_h^{n-2}}}{2 \Delta t},v_h\right) +  \calA_h(\widetilde{U_h^{n}},v_h) \nonumber \\
&\quad+\mathcal{C}(\widetilde{U_h^{n-2}})a_h^x(\widetilde{U_h^{n}},v_h)=  (g_h^{n},v_h)_h,\label{linear:discrete:model} \\
&\widetilde{U^0_h}:=I_h u_0, \qquad \text{and}~ \qquad \widetilde{U^1_h}:=\Delta tI_h \omega_0 +\widetilde{U^0_h}. \label{linear:discrete:initial} 
\end{align}
\eqref{linear:discrete:model} has the same matrix structure as
linear system of equation excluding the matrix $\bold{A}^x$
multiplied by a constant
$\mathcal{C}(\widetilde{U_h^{n-2}}):=\Big[ S \sum_{E \in \O_h}\int_{E} |\Pi^{k-1}_E D_x \widetilde{U_h^{n-2}}|^2-P\Big]$.
Recollecting the matrix representation of the bilinear forms  $m_h(\cdot,\cdot)$
and $\mathcal{A}_h(\cdot,\cdot)$ from the Section~\ref{semi:discrete:sec},
we rewrite~\eqref{linear:discrete:model}-\eqref{linear:discrete:initial} as follows:
\begin{align}
&\Big(\Big(1+ \frac{\delta~\Delta t}{2}\Big)\textbf{M}+\Delta t^2\textbf{A}+C~\Delta t^2\textbf{A}^x \Big) \widetilde{\Eta}^{n}
=\widetilde{G}(\widetilde{\Eta}^{n-1},\widetilde{\Eta}^{n-2}) \label{linear:matrix}\\
&\widetilde{\Eta}(0)=\widetilde{\Eta}_0 \label{linear:displacement}\\
&\frac{d \widetilde{\Eta}}{d t}=\boldsymbol{\omega}_0. \label{linear:velocity}
\end{align}   
Since the matrix $\bold{M}$ is positive definite and the matrices $\bold{A}$
and $\bold{A}^x$ are positive semi-definite, $\Big(\Big(1+ \frac{\delta~\Delta t}{2}\Big)\textbf{M}+\Delta t^2\textbf{A}+C~\Delta t^2\textbf{A}^x \Big)$ is invertible,
hence the algebraic system \eqref{linear:matrix}-\eqref{linear:velocity} has a unique solution.

\section{Convergence analysis of the discrete scheme}
\label{convergence:VEM}

In this section, we will derive the \textit{a priori} error estimates for semi-discrete,
fully-discrete and linearized schemes. We define the projection operator
(Ritz's projection) $R_h:H^2_{\ast}(\Omega) \rightarrow \mathcal{Z}_h$ such that
\begin{equation}
\calA_h(R_h u, v_h)=\calA(u,v_h) \quad \forall v_h \in \mathcal{Z}_h
\label{ritz:projection}
\end{equation}
Following \cite{AMNS-Submitted2020,VB15}, it can be prove that the discrete bilinear
form $\calA_h(\cdot,\cdot)$ is coercive and for any function $\omega\in H^2_{\ast}(\O)$,
$\calA(\omega,\cdot)$ is continuous on $\mathcal{Z}_h$.
By using the Lax-Milgram Theorem, we can conclude that \eqref{ritz:projection} has unique solution.
The existence and uniqueness of $R_h u$ directly follows from the fact that $R_h u$
is the solution of the variational problem~\eqref{ritz:projection}.
It can be perceived that employing the projection operator $R_h$,
we can bound the error $u-u_h$ easily. In this direction, we split the error as below.
\begin{equation}
u-u_h=u-R_hu+R_hu-u_h:=\varphi-\psi.
\label{splitting}
\end{equation}
By using the approximation property of  $R_h$ (Lemma~\ref{lemma:engy:pro}), we can bound $\varphi$.
To estimate $u-u_h$, we focus to estimate $\psi$. In this direction, we explore polynomial approximation
and interpolation operator properties on discrete space $\mathcal{Z}_h$
\cite{BS-2008,ABSV2016}.

\begin{proposition}
\label{app1}
(\textbf{Polynomial Approximation})
Assume that the mesh regularity assumption~\ref{mesh:regularity} is satisfied.
Then there exists a constant $C>0$ independent of mesh-size $h$
but depends on the mesh regularity parameter $\gamma$ 
such that for every $z\in H^{\delta}(E)$ there exists
$z_{\pi}\in\P_k(E)$, $k\in\mathbb{N}$ such that
\begin{equation*}
\vert z-z_{\pi}\vert_{\ell,E}\leq C h_E^{\delta-\ell}|z|_{\delta,E}\quad 0\leq\delta\leq
k+1, \ell=0,\ldots,[\delta],
\end{equation*}
with $[\delta]$ denoting largest integer equal or smaller than $\delta \in {\mathbb R}$.
\end{proposition}

\begin{proposition}\label{app2}
(\textbf{Interpolation Approximation}) Under the assumption of Proposition~\ref{app1} and 
for all $z\in H^s(E)$ there exists $I^E_h(z)\in \ZK$ and $C>0$
independent of $h$ such that
\begin{equation*}
||z-I^E_h(z) ||_{\ell,E}\leq C h_E^{s-\ell}|z|_{s,E},\quad \ell=0,1,2, \quad 2\leq s\leq k+1,
\end{equation*}
where $C$ is independent of mesh-size $h$ but depends on
mesh regularity parameter $\gamma$ (assumption~\ref{mesh:regularity}).
\end{proposition}
For each element $E\in \O_h$, we deduce that
\begin{equation*}
\dof_i(v_h)=\dof_i(I_h^E(v_h)) \quad \forall \ 1\leq i \leq N^{\dof}, 
\end{equation*}
where $v_h \in \ZK$ be an arbitrary element. The global interpolation is defined as
$I_h(v_h)|_E:=I_h^E(v_h)$.

The projection operator $R_h$ satisfies the following approximation properties.
\begin{lemma}
\label{lemma:engy:pro}
There exists a unique function $R_h(u)\in \mathcal{Z}_h$
such that the following approximation properties hold:
\begin{itemize}
\item[1)] There exists a positive constant $C$, independent of $h$, such that
\begin{equation*}
\|u-R_h(u)\|_{2,\O} \leq Ch^{\min\{s,k-1\}}|u|_{2+s,\O}
\quad  0\leq s \leq k-1.
\end{equation*}
\item[2)] There exist a positive constant $C$ and $\tilde{s} \in (1/2,1]$,
independent of $h$, such that
\begin{equation*}
\|u-R_h(u)\|_{1,\O} \leq Ch^{\tilde{s}+\min\{s,k-1\}}|u|_{2+s,\O}
\quad 0 \leq s \leq k-1.
\end{equation*}

\item[3)]  There exists a positive constant $C$, independent of $h$, such that
\begin{itemize}
\item[(a)] If $k=2$, then there exists $\tilde{s} \in (1/2,1]$,
independent of $h$, such that
\begin{equation*}
\|u-R_h(u)\|_{0,\Omega} \leq C h^{\tilde{s}+\min\{s,1\}}|u|_{2+s,\O}
\quad 0 \leq s\leq1.
\end{equation*}
\item[(b)] If $k\geq 3$, then there exist $s>1/2$ and
$\gamma \in (1/2,2]$, independent of $h$, such that
\begin{equation*}
\|u-R_h(u)\|_{0,\Omega} \leq C h^{\gamma+\min\{s,\,k-1\}}|u|_{2+s,\O}
\quad 0 \leq s \leq k-1.
\end{equation*}
\end{itemize}   

\end{itemize}   
\end{lemma}

\begin{proof}
The estimations of $u-R_h(u)$ in $\|\cdot\|_{2,\O}$ and $\|\cdot\|_{1,\O}$ norms can be proved
following analogous arguments as \cite{AMNS-Submitted2020}. Now, we proceed to prove the estimation
of $u-R_h(u)$ in $\|\cdot\|_{0,\O}$-norm.
We start with $(3a)$: the estimate is a direct consequence of the estimate (2)
and the Poincar\'e inequality.

Now, we continue with $(3b)$. Let $\phi \in H_{\ast}^2(\O)$ be the solution
of the following auxiliary variational problem:
\begin{equation}\label{aux-prob-L2}
\mathcal{A}(\phi,v)=\int_{\O} (u-R_h(u))  v \quad  \forall v \in H_{\ast}^2(\O) ,
\end{equation}
where $\mathcal{A}(\cdot, \cdot)=(\cdot,\cdot)_{H_{\ast}^2(\O)}$ (cf. \eqref{inner:H2st}).
As a consequence of a classical regularity result for the biharmonic
problem, there exists $\gamma \in (1/2,2]$
such that $\phi \in H^{2+\gamma}(\O) $ and  
\begin{equation*}
\|\phi\|_{2+\gamma, \O} \leq  C\|u-R_h(u)\|_{0,\O}.
\end{equation*}
Next, by standard duality arguments, we get
\begin{equation*}
\|u-R_h(u)\|_{0,\O}\leq  Ch^{\gamma+ \min\{s, \, k-1\}}|u|_{2+s,\O}, \quad k \geq 3.
\end{equation*}
The proof is complete.
\end{proof}

With the help of approximation property of the projection operator $R_h$,
we move to estimate the bound for the nonlinear term as follows.
\begin{lemma}
\label{lem:nonlocal:difference}
Let $u \in H^2_{\ast}(\O)$ be the solution of \eqref{weak_sol} and let $R_h$
be the Ritz-projection operator defined in \eqref{ritz:projection}.
Then, there exists a positive constant $C$ depends on mesh regularity of
Assumption~\ref{mesh:regularity} and Sobolev regularity of $u$, independent of $h$, such that
\begin{equation}
\begin{split}
\Big |\Big ( S \sum_{E \in \O_h} \int_{E} |\Pi^{k-1}_E D_x  u_h|^2 -P \Big)
&a_h^x(R_hu,v_h)-\Big (S \sum_{E \in \O_h} \int_{E} | D_x u|^2 -P \Big) a^x(u,v_h) \Big |\\
&\leq C(S,P)~h^{k}~ |u|_{k+1,\O}~\|D_xv_h\|_{0,\Omega}\\
& \quad+C~S~\|D_x(u_h-u)\|_{0,\O}~\|D_xv_h\|_{0,\O} \quad \forall  v_h \in \mathcal{Z}_h.
\end{split}
\label{nonlocal:difference}
\end{equation}
\end{lemma}

\begin{proof}
We have that
\begin{equation}
\begin{split}
&\Big [S \sum_{E \in \O_h}\int_{E} |\Pi^{k-1}_E D_x  u_h|^2 -P \Big] a_h^x(R_hu,v_h)-\Big [ S \sum_{E \in \O_h} \int_{E} |D_x u|^2 -P \Big] a^x(u,v_h)\\
& =\Big(S \sum_{E \in \O_h}\int_{E} |\Pi^{k-1}_E D_x  u_h|^2 a_h^x(R_hu,v_h)-S \sum_{E \in \O_h} \int_{E} | D_x u|^2 a^x(u,v_h) \Big )\\
& \quad + P\Big(a^x(u,v_h)-a_h^x(R_hu,v_h) \Big).
\end{split}
\label{proof:nonlocal:1}
\end{equation}  
Further, the first term on the right-hand side of \eqref{proof:nonlocal:1} can be written as 
\begin{equation*}
\begin{split}
&\Big(S\sum_{E \in \O_h} \int_{E} |\Pi^{k-1}_E D_x u_h|^2 a_h^x(R_hu,v_h)-S \sum_{E  \in \O_h}\int_{E} | D_x u|^2 a^x(u,v_h) \Big)\\
& =\Big( S \sum_{E \in \O_h} \int_{E} |\Pi^{k-1}_E D_x  u_h|^2-S \sum_{E \in \O_h}\int_{E} | D_x u|^2 \Big) a_h^x(R_hu,v_h)\\
& \quad + S \sum_{E \in \O_h}\int_{E} | D_x u|^2 \Big(a_h^x(R_hu,v_h)-a^x(u,v_h) \Big).
\end{split}
\end{equation*}

By using simple algebra and boundedness of discrete solution $u_h$,
regularity of $u$ and boundedness of the projection operator $\Pi^{k-1}_E$, we obtain
\begin{equation}
\begin{split}
&\Big( S \sum_{E \in \O_h}\int_{E} |\Pi^{k-1}_E D_x  u_h|^2-S \sum_{E \in \O_h} \int_{E} | D_x u|^2 \Big) a_h^x(R_hu,v_h)\\
&\leq S\sum_{E\in \O_h} \int_{E} \Big(|\Pi^{k-1}_E D_x  u_h|+| D_x u| \Big) ~\Big(|\Pi^{k-1}_E D_x  u_h|-| D_x u| \Big) \| D_xR_h u\|_{0,\O}~\|D_x v_h\|_{0,\O}\\
& \leq C~S~ \sum_{E \in \O_h}(\| \Pi^{k-1}_E D_x u_h\|_{0,E}+\|D_x u\|_{0,E})(\| \Pi^{k-1}_E D_x u_h-D_x u\|_{0,E})~\| D_xR_h u\|_{0,\O}~\|D_x v_h\|_{0,\Omega}\\ 
& \leq C~S (\sum_{E \in \O_h}(\| \Pi^{k-1}_E D_x u_h\|_{0,E}+\|D_x u\|_{0,E})^2)^{1/2}~(\sum_{E \in \O_h}(\| \Pi^{k-1}_E D_x u_h-D_x u\|_{0,E})^2)^{1/2} \\
& \quad\|D_x R_h u\|_{0,\O}~\|D_x v_h\|_{0,\O} \\
&\leq C(u_h)~S~ \Big(( \sum_{E \in \O_h}\|\Pi^{k-1}_E D_x u_h-\Pi^{k-1}_E D_x u\|_{0,E}^2)^{1/2}+(\sum_{E \in \O_h}\|\Pi^{k-1}_E D_x u-D_xu\|_{0,E}^2)^{1/2} \Big)\\
& \quad \|D_x R_h u \|_{0,\O}~\|D_x v_h\|_{0,\O}\\
&\leq C~S~ \big( |u_h-u|_{1,\O}+h^k |u|_{k+1,\O}\big)\|D_xR_h u\|_{0,\O}~\|D_xv_h\|_{0,\O}.
\end{split}
\label{proof:nonlocal:2}
\end{equation}

 Using the definition of discrete bilinear form $a_h^x(\cdot,\cdot)$ and approximation
 property of projection operator $R_h$, and polynomial approximation property of
 $a_h^x(\cdot,\cdot)$ (cf. Lemma~\ref{poly:consistency}), we derive 
 \begin{equation}
 \begin{split}
  \Big(a_h^x(R_hu,v_h)-a^x(u,v_h) \Big)&= \sum_{E\in \O_h}\Big(a_h^{x,E}(R_hu-u_{\pi},v_h)-a^{x,E}(u-u_{\pi},v_h) \Big)\\
 & \leq\sum_{E \in \O_h} (\|D_x(R_hu-u_{\pi})\|_{0,\O}+\|D_x(u-u_{\pi})\|_{0,E})\|D_xv_h\|_{0,E} \\
 &\leq C h^k |u|_{k+1,\O} \|D_xv_h\|_{0,\O}.
 \end{split}
 \label{proof:nonlocal:3}
 \end{equation}
Inserting the estimations \eqref{proof:nonlocal:2} and \eqref{proof:nonlocal:3} into \eqref{proof:nonlocal:1}, we derived the intended result.
 
\end{proof}

\subsection{Error estimates for semi-discrete scheme}
In this section, we will derive the error estimation for the semi-discrete scheme (cf. \eqref{semi_dis_schm:1}-\eqref{semi_dis_schm:3}). With
this end, we state the following theorem.
\begin{theorem}
\label{semi_dis:error}
Let $u$ be the solution of \eqref{weak_sol} and $u_h$ be the semi-discrete
solution of \eqref{semi_dis_schm:1}-\eqref{semi_dis_schm:3}. Let us assume that the Assumption~\ref{mesh:regularity}
holds. Furthermore, we assume that the exact solution $u$ satisfies following regularity
$u \in L^2(0,T;H^{k+1}(\Omega))$, $D_t u \in L^2(0,T;H^{k+1}(\Omega))$ and the force function
$g \in L^2(0,T;H^{k+1}(\Omega))$ and $u_{h,0}:=I_h u_0 \in \mathcal{Z}_h$
and $\omega_{h,0}:=I_h \omega_0 \in \mathcal{Z}_h $ be the initial approximation
of $u_0$ and $\omega_0$, respectively. Then, there exists a positive constant $C$
independent of mesh size $h$ but depends on mesh-regularity parameter $\gamma$,
Sobolev regularity of $u$, stability parameter of bilinear forms $\mathcal{A}_h(\cdot,\cdot)$,
and $m_h(\cdot,\cdot)$, and continuity of $a_h^x(\cdot,\cdot)$ such that
for all $t \in (0,T]$, the following regularity holds
\begin{equation*}
\begin{split}
&\|D_t(u-u_h)(t)\|_{0,\O}+\|u(t)-u_h(t)\|_{2,\O} \leq C(\TGstr,\TAstr)(\|D_t(u-u_h)(0)\|_{0,\O}+\|u(0)-u_h(0)\|_{2,\O})\\
& \quad +C(\TGstr,\TAstr,\TGstR,\delta) h^{k-1} \Big( |\omega_0|_{k+1,\O}+|u_0|_{k+1,\O}+\|g\|_{L^2(0,T;H^{k+1}(\O))}+\|D_{tt}u\|_{L^2(0,T;H^{k+1}(\O))} \\
& \quad +\|D_t u\|_{L^2(0,T;H^{k+1}(\O))}+\|u\|_{L^2(0,T;H^{k+1}(\O))} \Big).
\end{split}
\end{equation*}
\end{theorem}
\begin{proof}
Using weak formulation~\eqref{weak_sol}, semi-discrete formulation~\eqref{semi_dis_schm:1},
definition of $R_h$ in \eqref{ritz:projection}, and the splitting of $u-u_h$ as \eqref{splitting}, we obtain 
\begin{equation}
\begin{split}
&m_h(D_{tt} \psi(t), v_h)+ \delta m_h(D_t \psi(t),v_h)+\calA_h(\psi(t),v_h)+\Big[S\sum_{E \in \O_h} \int_{E}|\Pi^{k-1}_E D_x u_h|^2-P \Big]a_h^x(\psi(t),v_h)\\
&=(g_h,v_h)_h
-m_h(D_{tt} R_hu(t),v_h)-\delta m_h(D_t R_hu(t),v_h)-\calA_h(R_h u(t),v_h)\\
& \quad -\Big[S \sum_{E \in \O_h}\int_{E}|\Pi^{k-1}_E D_x u_h|^2-P \Big]a_h^x(R_hu(t),v_h) \\
&=(g_h,v_h)_h-(g,v_h)-m_h(D_{tt} R_hu(t),v_h)+(D_{tt} u(t),v_h)-\delta m_h(D_t R_hu(t),v_h)\\
& \quad+\delta (D_t u(t),v_h)-\Big[S \sum_{E \in \O_h}\int_{E}|\Pi^{k-1}_E D_x u_h|^2-P \Big] a_h^x(R_hu,v_h) \\
& \quad +\Big[S \sum_{E \in \O_h} \int_{E}| D_x u|^2-P \Big]~a^x(u,v_h).
\end{split}
\label{semi_dis:eqn1}
\end{equation}

By using the approximation property of the projection operator $\Pi^{k}_E$
and Cauchy-Schwarz inequality, we bound the load term as below
\begin{equation}
|(g_h,v_h)_h-(g,v_h)| \leq  \sum_{E \in \O_h} \|\Pi^k_E g-g\|_{0,E} \|v_h\|_{0,E} \leq C h^{k+1}|g|_{k+1,\O}~\|v_h\|_{0,\O}.
\label{semi_dis:eqn2}
\end{equation}

An application of approximation property of  Ritz  operator $R_h$ (cf. Lemma~\ref{lemma:engy:pro})
and polynomial consistency of $m_h(\cdot,\cdot)$ (cf. Lemma~\ref{poly:consistency}), continuity of $m_h(\cdot,\cdot)$
and polynomial approximation property (cf. Lemma~\ref{app1}) yields the bound

\begin{equation}
\begin{split}
|m_h(D_{tt} R_hu,v_h)-(D_{tt} u,v_h)|&\leq |m_h(D_{tt} R_hu-D_{tt}u_{\pi},v_h)|+|(D_{tt}u_{\pi}-D_{tt}u,v_h)|\\
&\leq C(\TGstR) \sum_{E \in \O_h} (\|D_{tt} R_hu-D_{tt}u_{\pi}\|_{0,E}+\|D_{tt}u_{\pi}-D_{tt}u\|_{0,E})\|v_h\|_{0,E} \\
&\leq C(\TGstR)~h^{k+1}|D_{tt}u|_{k+1,\O} \|v_h\|_{0,\O}.
\end{split}
\label{semi_dis:eqn3}
\end{equation} 
In the above estimation, we have used the property that $R_h$ commutes with time-derivative and $C(\TGstR)$ is a positive generic constant.

By using analogous arguments as \eqref{semi_dis:eqn3}, we bound the term
\begin{equation}
\delta~|m_h(D_{t} R_hu,v_h)-(D_{t} u,v_h)| \leq C(\TGstR,\delta)h^{k+1} |D_t u|_{k+1,\O}~\|v_h\|_{0,\O}.
\label{semi_dis:eqn4}
\end{equation} 

Further, using Lemma~\ref{lem:nonlocal:difference},
we estimate the error for nonlinear term as below
\begin{equation}
\begin{split}
& \Big |\Big[S \sum_{E \in \O_h}\int_{E}|\Pi^{k-1}_E D_x u_h|^2-P \Big] a_h^x(R_hu,v_h)-\Big[S \sum_{E \in \O_h} \int_{E}| D_x u|^2-P \Big]~a^x(u,v_h) \Big |\\
& \leq C(S,P)~ h^k |u|_{k+1,\O}~\|D_xv_h\|_{0,\O}+C~S~ \|D_x\psi\|_{0,\O}~\|D_xv_h\|_{0,\O}.   
\end{split}
\label{semi_dis:eqn5}
\end{equation}

Inserting \eqref{semi_dis:eqn2}-\eqref{semi_dis:eqn5} into \eqref{semi_dis:eqn1},
substituting the test function $v_h:=D_t\psi$, using continuity and stability
of $a_h^x(\cdot,\cdot)$, and assumption on $P$ we derive
\begin{equation}
\begin{split}
&\frac{1}{2} D_t m_h(D_t \psi,D_t \psi)+\delta m_h(D_t \psi,D_t \psi)+\frac{1}{4} D_t\mathcal{A}_h(\psi,\psi)+ \frac{1}{2} \Big(\frac{\TAstr~\lambda_1}{2}-P \Big)D_t\|D_x \psi(t)\|_{0,\O}^2 \\
&\quad + \frac{1}{2}\Big(S \sum_{E\in \O_h} \|\Pi^{k-1}_ED_x u_h\|_{0,E}^2 \Big) D_t a_h^x(\psi,\psi)\\
&  \leq C h^{k+1} (|g|_{k+1,\O}+|D_{tt}u|_{k+1,\O}+|D_{t}u|_{k+1,\O})\|D_t \psi\|_{0,\O} +C(S,P) h^k~|u|_{k+1,\O}|D_t \psi|_{1,\Omega} \\
& \quad +C~S~\|D_x\psi\|_{0,\Omega}~ \|D_t\psi\|_{1,\O}.
\end{split}
\label{semi_dis:eqn6}
\end{equation}
By using stability property of the discrete bilinear forms $m_h(\cdot,\cdot)$, $\mathcal{A}_h(\cdot,\cdot)$ (cf. Lemma~\ref{stability}),
and Young's inequality, and neglecting the term $\frac{1}{2}\Big(S \sum_{E\in \O_h} \|\Pi^{k-1}_ED_x u_h\|_{0,E}^2 \Big) D_t a_h^x(\psi,\psi)$, we obtain
\begin{equation}
\begin{split}
&\frac{1}{2}~\TGstr D_t\|D_t \psi(t)\|_{0,\O}^2+ \delta \TGstr~\|D_t \psi\|_{0,\O}^2 +\frac{1}{4} \TAstr D_t \|\psi(t)\|_{2,\O}^2
+\frac{1}{2} \Big(\frac{\TAstr~\lambda_1}{2}-P \Big) D_t\|D_x \psi(t)\|_{0,\O}^2\\
&\leq  C h^{2k}(|g|_{k+1,\O}^2+|D_{tt}u|_{k+1,\O}^2+|D_{t}u|_{k+1,\O}^2)+C \|D_t\psi\|_0^2+C(S,P)~h^{k}~|u|_{k+1,\O}|D_t \psi|_{1,\Omega}\\
& \quad +CS~\|D_x\psi\|_{0,\Omega}~ |D_t\psi|_{1,\O}.
\end{split}
\label{semi_dis:eqn7}
\end{equation}

Again, by applying Young's inequality and since the term $|D_t \psi|_{1,\Omega}$ is bounded,
we can build the term  $C(\epsilon,S^2,P^2)~|D_t \psi|_{1,\Omega}^2 $, which can be absorbed
by $\delta \TGstr \|D_t \psi\|_{0,\Omega}^2$ where $\epsilon$ is small positive parameter and we rewrite \eqref{semi_dis:eqn7} as follows

  \begin{equation}
\begin{split}
&\frac{1}{2}~\TGstr D_t\|D_t \psi(t)\|_{0,\O}^2+ \Big ( \delta \TGstr~\|D_t \psi\|_{0,\O}^2-C(\epsilon,S^2,P^2)~|D_t \psi|_{1,\Omega}^2 \Big ) +\frac{1}{4} \TAstr D_t \|\psi(t)\|_{2,\O}^2\\
& \quad  +\frac{1}{2} \Big(\frac{\TAstr~\lambda_1}{2}-P \Big) D_t\|D_x \psi(t)\|_{0,\O}^2\\
&\leq  C h^{2k}(|g|_{k+1,\O}^2+|D_{tt}u|_{k+1,\O}^2+|D_{t}u|_{k+1,\O}^2)+C \|D_t\psi\|_0^2+C~ h^{2k}~|u|_{k+1,\O}^2\\
& \quad +C~\|D_x\psi\|_{0,\Omega}^2.
\end{split}
\label{semi_dis:eqn71}
\end{equation}

Upon integrating both sides of \eqref{semi_dis:eqn71} with respect to $t$, and exploiting Gronwall inequality, we derive
\begin{equation}
\begin{split}
&\TGstr \|D_t \psi(t)\|_{0,\O}^2+\TAstr \|\psi(t)\|_{2,\O}^2\leq C(\TGstr,\TAstr)\Big (\|D_t \psi(0)\|_{0,\O}^2+\|\psi(0)\|_{2,\O}^2 \Big) \\
& \quad +C(\TGstr,\TAstr,\TGstR,\delta) h^{2k} \Big(\|g\|_{L^2(0,T;H^{k+1}(\O))}^2+\|D_{tt}u\|_{L^2(0,T;H^{k+1}(\O))}^2+\|D_t u\|_{L^2(0,T;H^{k+1}(\O))}^2 \\
& \quad+\|u\|_{L^2(0,T;H^{k+1}(\O))}^2\Big).
\end{split}
\label{semi_dis:eqn8}
\end{equation}

An application of approximation property of Ritz operator $R_h$ (cf. Lemma~\ref{lemma:engy:pro}), we derive as below
\begin{equation*}
\begin{split}
&\|D_t(u-u_h)(t)\|_{0,\O}+\|u(t)-u_h(t)\|_{2,\O} \leq C(\TGstr,\TAstr) \Big( \|D_t(u-u_h)(0)\|_{0,\O}+\|u(0)-u_h(0)\|_{2,\O}\Big) \\
&\quad +C(\TGstr,\TAstr,\TGstR,\delta) h^{k-1} \Big(|\omega_0|_{k+1,\Omega}+|u_0|_{k+1,\Omega}+\|g\|_{L^2(0,T;H^{k+1}(\O))}+\|D_{tt}u\|_{L^2(0,T;H^{k+1}(\O))} \\
& \quad +\|D_t u\|_{L^2(0,T;H^{k+1}(\O))}+\|u\|_{L^2(0,T;H^{k+1}(\O))} \Big).
\end{split}
\end{equation*}

\end{proof}
\begin{remark}
Using Young's inequality, we have chosen the coefficient in \eqref{semi_dis:eqn71} $C(\epsilon,S^2,P^2)$ small enough such that $C(\epsilon,S^2,P^2) |D_t \psi|_{1,\Omega}^2$ can be absorbed by the term $\delta~\widetilde{\gamma_{\ast}} \|D_t \psi \|_{0,\Omega}^2$. A straight forward calculation infer the choice of $\epsilon$ which should be satisfied
$$ \epsilon \leq \frac{\delta~\widetilde{\gamma_{\ast}}~\|D_t \psi\|^2_{0,\Omega}}{C(S^2,P^2)~|D_t \psi|_{1,\Omega}^2}.$$
Since the coefficient $1/C(S^2,P^2)$  involving $S^2$ and $P^2$ is a sufficiently big quantity, we can choose $\epsilon$ such that the coefficient in right-hand side of \eqref{semi_dis:eqn71} involving $1/\epsilon$ is a bounded quantity which does not affect the optimal order of convergence as stated in Theorem~\ref{semi_dis:error}.
\end{remark}

\subsection{Error estimates for fully-discrete scheme}
In this section, we would like to study the convergence analysis of the fully-discrete scheme.
With this aim, for each time-step $t_n$, we denote by $u^n:=u(t_n)$ $1\leq n\leq N$.
We divide the error as below
\begin{equation*}
u^n-U_h^n=u^n-R_hu^n+R_hu^n-U_h^n:=\varphi^n-\psi^n
\end{equation*}
By utilizing the approximation property of  $R_h$ at time $t=t_n$, we can estimate $\varphi^n$.
Therefore, we focus on the bound of $\psi^n$. Further, to present the analysis ambiguously, we introduce the following notation:
 \begin{equation*}
d_t^2 \psi^n:=\frac{\psi^{n}-2\psi^{n-1}+\psi^{n-2}}{\Delta t^2};
\qquad \partial_t \psi^n:=\frac{\psi^{n}-\psi^{n-2}}{2 \Delta t}.
\end{equation*}

\begin{theorem}
\label{fulldis:mainth}
Let $u \in H^2_{\ast}(\O)$ be the solution of \eqref{weak_sol} and let $U_h^n \in \mathcal{Z}_h$
be the solution of \eqref{full:discrete:model} for time $t=t_n$, where $1\leq n\leq N$.
Further, assume that the Assumptions~\ref{mesh:regularity}, and assumption of
Theorem~\ref{wellposed:full:discrete}  satisfy and
$\Big \|\frac{\psi^1-\psi^0}{\Delta t}  \Big \|_{0,\Omega}=O(h^{k+1}+\Delta t^2)$
\cite[Theorems~3 and 4]{D73}.
Then under the assumption of Theorem~\ref{semi_dis:error}, there exists a positive generic constant $C$
that depends on mesh regularity parameter $\gamma$, Sobolev regularity of $u$, stability parameters
of bilinear forms $\mathcal{A}_h(\cdot,\cdot)$, and $m_h(\cdot,\cdot)$, and continuity of $a_h^x(\cdot,\cdot)$
but independent of mesh size $h$ and time-step $\Delta t$ such that the following estimation holds
\begin{equation*}
\|u^n-U_h^n\|_{2,\O} \leq C(\TGstr,\TAstr,\TGstR, u,D_{ttt}u,D_{tt}u,D_t u,g,d ) ~\Big(h^{k-1}+\Delta t^2 \Big).
\end{equation*}  
\end{theorem}
\begin{proof}
By using the fully-discrete scheme \eqref{full:discrete:model}, weak formulation \eqref{weak_sol} and \eqref{ritz:projection}, we derive
\begin{equation}
\begin{split}
&m_h(d_t^2 \psi^n,v_h)+\delta m_h(\partial_t \psi^n,v_h)+\mathcal{A}_h(\psi^{n},v_h)+\Big[S \sum_{E \in \O_h}\int_{E}|\Pi^{k-1}_E D_x U_h^{n}|^2-P \Big]a_h^x(\psi^{n},v_h)\\
&=(g^{n}_{h},v_h)_h-m_h(d_t^2 R_h u^n,v_h)-\delta m_h(\partial_t R_h u^n,v_h)-\mathcal{A}_h(R_h u^{n},v_h)\\
& \quad -\Big[S \sum_{E \in \O_h}\int_{E} |\Pi^{k-1}_E D_x U_h^{n}|^2-P \Big] a_h^x(R_hu^{n},v_h)\\
&=(g_h^{n},v_h)_h-(g^{n},v_h)-m_h(d_t^2 R_h u^n,v_h)+(D_{tt}u^{n},v_h)\\
& \quad -\delta m_h(\partial_tR_h u^n,v_h ) +\delta (D_t u^n,v_h)-\mathcal{A}_h(R_h u^{n},v_h)+\mathcal{A}(u^{n},v_h)\\
&\quad-\Big[S\sum_{E \in \O_h} \int_{E} |\Pi^{k-1}_E D_x U_h^{n}|^2-P \Big] a_h^x(R_hu^{n}_h,v_h)+\Big[S \sum_{E \in \O_h} \int_{E}| D_xu^{n}|^2-P \Big] a^x(u^{n},v_h).
\end{split}
\label{fulldis:proof:1}
\end{equation}

By using approximation property of the projection operator $\Pi^{k}$ at time $t=t_n$, we derive
\begin{equation}
|(g_h^{n},v_h)_h-(g^{n},v_h)|\leq C h^{k+1}|g^{n}|_{k+1,\O}\|v_h\|_{0,\O}.
\label{fulldis:proof:2}
\end{equation}
Next, we split the third and fourth terms on the right hand side in \eqref{fulldis:proof:1} as follows
\begin{equation}
m_h(d_t^2 R_h u^n,v_h)-(D_{tt}u^{n},v_h)=m_h(d_t^2 R_h u^n,v_h)-(d_t^2u^n,v_h)+(d_t^2u^n,v_h)-(D_{tt}u^{n},v_h)
\end{equation}
Using approximation property of $R_h$ and following the analogous technique as \cite{AN-MCS2020,VB15}, we estimate as 
\begin{equation}
|m_h(d_t^2 R_h u^n,v_h)-(d_t^2u^n,v_h)|\leq Ch^{k+1}|d_t^2 u^n|_{k+1,\O} \|v_h\|_{0,\O}.
\end{equation}

Using Taylor's series expansion and fundamental theorem of calculus \cite{D73,AN-MCS2020}, we bound 
\begin{equation}
|(d_t^2u^n,v_h)-(D_{tt}u^{n},v_h)|\leq C \Delta t~\|D_{ttt}u^{n-1}\|_{0,\O}\|v_h\|_{0,\O}.
\end{equation}

By emphasizing on identical techniques as \cite{AN-NMPDE2019,VB15}, we bound the following term as below
\begin{equation*}
\begin{split}
|-\delta m_h(\partial_tR_h u^n,v_h ) +\delta (D_t u^n,v_h)|&\leq
C(\delta)~\Big(h^{k+1}~|\partial_t u^n|_{k+1,\O}\|v_h\|_{0,\O} \\
&+\Delta t^2~\|D_{tt} u^{n}\|_{0,\O} \|v_h\|_{0,\O} \Big).
\end{split}
\end{equation*}

By utilizing approximation property of the Ritz projection operator $R_h$, polynomial consistency property,
stability and continuity of $\mathcal{A}_h(\cdot,\cdot)$ and \eqref{ritz:projection}, we derive
\begin{equation}
|\mathcal{A}_h(R_hu^{n},v_h)-\mathcal{A}_h(u^{n},v_h)|\leq C(\TAstR) h^{k+1}~|u^{n}|_{k+1,\O}~\|v_h\|_{2,\O}.
\end{equation}
Now, we proceed to bound the nonlocal term. Using Lemma~\ref{lem:nonlocal:difference} at time $t=t_{n}$, we obtain 
\begin{equation}
\begin{split}
&-\Big[S \sum_{E \in \O_h}\int_{E} |\Pi^{k-1}_E D_x U_h^{n}|^2-P \Big] a_h^x(R_hu^{n}_h,v_h)+\Big[S \sum_{E \in \O_h}\int_{E} |D_xu^{n}|^2-P \Big] a^x(u^{n},v_h)\\ 
& \leq C(S,P)~ h^{k}~|u^{n}|_{k+1,\O}~\|D_x v_h\|_{0,\O}  + C~S~\|D_x\psi^{n}\|_{0,\O}~\|D_x v_h\|_{0,\O} .
\end{split}
\label{fulldis:proof:3}
\end{equation}

By choosing $v_h=\partial_t \psi^n$ into \eqref{fulldis:proof:1} and inserting \eqref{fulldis:proof:2}-\eqref{fulldis:proof:3}
into \eqref{fulldis:proof:1}, and using third inequality in \eqref{euthdgfk}, we obtain
\begin{equation}
\begin{split}
& \frac{C(\TGstr)}{2} \Big \|\frac{\psi^{n}-\psi^{n-1}}{\Delta t} \Big \|^2_{0,\O} -C(\TGstR) \Big\|\frac{\psi^{n-1}-\psi^{n-2}}{\Delta t} \Big \|^2_{0,\O} 
+ \delta~ \TGstr~ \Delta t~\|\partial_t \psi^n\|^2_{0,\O}+ \frac{\TAstr}{2 } \|\psi^n\|_{2,\O}^2 \\
& \quad +\left ( \frac{\TAstr~\lambda_1}{2}-P \right) \|D_x \psi^n\|_{0,\O}^2+ \left ( S \sum_{E \in \Omega_h} \|\Pi^{k-1}_E D_x \psi^n\|^2_{0,E} \right ) \|D_x \psi^n\|_{0,\O}^2 \\
&\leq  |\calA_h(\psi^n,\psi^{n-2})|+m_0~ |a_h^x(\psi^n,\psi^{n-2})|+C ~h^{2k}\Big(|g^{n}|_{k+1,\O}^2+ |d_t^2u^n|_{k+1,\O}^2+|\partial_tu^n|_{k+1,\O}+|u^{n}|_{k+1,\O}^2\Big ) \\
&  \quad +C \Delta t^4 \Big( \|D_{ttt} u^{n-1}\|_{0,\O}^2+ \|D_{tt}u^{n}\|_{0,\O}^2 \Big)
+C \Delta t~\left ( \Big \|\frac{\psi^{n}-\psi^{n-1}}{\Delta t}\Big \|^2_{0,\Omega}+ \Big \|\frac{\psi^{n-1}-\psi^{n-2}}{\Delta t} \Big \|^2_{0,\O} \right )\\
& \quad + C~S~ \|\psi^n\|_{2,\O}^2+C~\|D_x \psi^{n-2}\|^2_{0,\Omega} +C~ \Big \|\frac{\psi^{n-1}-\psi^{n-2}}{\Delta t} \Big \|^2_{0,\O}.
\end{split}
\label{fulldis:proof:4}
\end{equation}

Using Young's inequality, and assumption on $P$ (cf. Theorem~\ref{wellposed:full:discrete}),
and neglecting the positive term $\left ( S \sum_{E \in \Omega_h} \|\Pi^{k-1}_E D_x \psi^n\|^2_{0,E} \right ) \|D_x \psi^n\|_{0,\O}^2$, we obtain 

\begin{equation}
\begin{split}
& \frac{C(\TGstr)}{2} \Big \|\frac{\psi^{n}-\psi^{n-1}}{\Delta t} \Big \|^2_{0,\Omega}- C(\TGstR) \Big\|\frac{\psi^{n-1}-\psi^{n-2}}{\Delta t} \Big \|^2_{0,\Omega} + \delta~ \TGstr~ \Delta t~\|\partial_t \psi^n\|^2_{0,\Omega}+ \frac{\TAstr}{4 } \|\psi^n\|_{2,\Omega}^2 \\
& \quad +1/2 \left ( \frac{\TAstr~\lambda_1}{2}-P \right) \|D_x \psi^n\|_{0,\Omega}^2 \\
&\leq C(\TAstr,\TAstR) \|\psi^{n-2}\|_{2,\Omega}^2+C(m_0,\TAstr)~\|D_x \psi^{n-2}\|_{0,\Omega}^2 +C ~h^{2k}\Big(|g^{n}|_{k+1,\Omega}^2 \\
& \quad + |d_t^2u^n|_{k+1,\Omega}^2+|\partial_tu^n|_{k+1,\Omega} +|u^{n}|_{k+1,\Omega}^2\Big ) +C \Delta t^4 \Big( \|D_{ttt} u^{n-1}\|^2_{0,\Omega}+ \|D_{tt}u^{n}\|^2_{0,\Omega} \Big) \\
& \quad +C \Delta t~\left ( \Big \|\frac{\psi^{n}-\psi^{n-1}}{\Delta t}\Big \|^2_{0,\Omega}+ \Big \|\frac{\psi^{n-1}-\psi^{n-2}}{\Delta t} \Big \|^2_{0,\Omega} \right )+C~S~\|\psi^n\|_{2,\Omega}^2+C~ \Big \|\frac{\psi^{n-1}-\psi^{n-2}}{\Delta t} \Big \|^2_{0,\O}.
\end{split}
\label{fulldis:proof:4:updated}
\end{equation}

Following \cite{AN-MCS2020}, we write the term as
\begin{equation*}
d_t^2 u^n:=\frac{1}{\Delta t^2} \int_{-\Delta t}^{\Delta t} (\Delta t -|\tau|) D_{tt}u^{n-1}(t_{n-1}+\tau)~ {\rm d} \tau,
\end{equation*}
which implies 
\begin{equation}
\Delta t \sum_{n=2}^{N} \|d_t^2 u^n\|_{k+1,\O}^2 \leq C \|D_{tt}u\|_{L^2(0,T;H^{k+1}(\Omega))}.
\label{fulldis:proof:5}
\end{equation}

Upon iterating \eqref{fulldis:proof:4} $n=2$ to $n$ and using discrete Gronwall
inequality and \eqref{fulldis:proof:5}, and for sufficiently small values of $S$ ($C~S <1$)
where $C~S$ denote the coefficient of the sum $\sum_{j=1}^{n} \|\psi^j\|_{2,\Omega}$, we derive
\begin{equation*}
\begin{split}
 \|\psi^{n} \|_{2,\Omega} &\leq C\Big( \Big \|\frac{\psi^1-\psi^0}{\Delta t} \Big \|_{0,\Omega}+\|\psi^0\|_{2,\Omega}+\|\psi^1\|_{2,\Omega} \Big)+ C h^{k} \Big(|g|_{L^{\infty}(0,T;H^{k+1}(\Omega))} +\|D_{tt}u\|_{L^2(0,T;H^{k+1}(\Omega))}\\
& \quad+\|D_{t}u\|_{L^2(0,T;H^{k+1}(\Omega))}+\|u\|_{L^{\infty}(0,T;H^{k+1}(\Omega))}\Big)+C\Delta t^2 \Big(\|D_{ttt} u\|_{L^{\infty}(0,T;L^2(\Omega))}\\
& \quad + \|D_{tt} u\|_{L^{\infty}(0,T;L^2(\Omega))} \Big). 
\end{split}
\end{equation*}

By using the assumption of Theorem~\ref{fulldis:mainth} and approximation property of $R_h$
at time $t=t_n$ (cf. Lemma~\ref{lemma:engy:pro}), approximation property of interpolation operator $I_h$, we obtain the intended result.
\begin{equation*}
\|U_h^n-u^n\|_{2,\Omega}\leq C(h^{k-1}+\Delta t^{2}).
\end{equation*}
\end{proof}
\begin{remark} 
After simplifying equation~(4.68), the coefficient involving the term $\|\frac{\psi^n-\psi^{n-1}}{\Delta t}\|^2_{0,\Omega}$ contains $\widetilde{\gamma_{\ast}},\delta, \widetilde{\alpha^{\ast}}$ and $\Delta t$. The coefficient involving the term $\|\frac{\psi^n-\psi^{n-1}}{\Delta t}\|^2_{0,\Omega}$ contains $\widetilde{\gamma_{\ast}},\delta, \widetilde{\alpha^{\ast}}$. Therefore, after iterating the inequality (4.68) $n=2$ to $n$, the coefficient of $ \sum_{j=2}^n\|\frac{\psi^j-\psi^{j-1}}{\Delta t}\|^2_{0,\Omega}$ contains term involving $\widetilde{\gamma_{\ast}}^{n-1},\delta^{n-1}, \widetilde{\alpha^{\ast}}^{n-1}$ and $\Delta t^{n-1}$ and the coefficient of $\|\frac{\psi^{1}-\psi^{0}}{\Delta t}\|^2_{0,\Omega}$ contains $\widetilde{\gamma_{\ast}}^{n-1},\delta^{n-1}, \widetilde{\alpha^{\ast}}^{n-1}$. Both the coefficients are bounded and by using discrete Gronwall inequality, we can achieve the desired result.
\end{remark}

In Section~\ref{section:linearized}, we have proposed linearized scheme~\eqref{linear:discrete:model}
and  highlighted that the scheme provides optimal order of convergences for both space and time variables.
Next, we proceed to prove in the following theorem
\begin{theorem}
Let $u \in H_{\ast}^{2}(\Omega)$ be the solution of \eqref{weak_sol} and let $\widetilde{U_h^{n}} \in \mathcal{Z}_h$
be the solution of \eqref{linear:discrete:model} at time $t=t_n$, where $2 \leq n \leq N$.
Then, under the assumptions of Theorem~\ref{fulldis:mainth}, there exists a constant $C$
which  depends on mesh regularity parameter $\gamma$, Sobolev regularity of $u$,
stability parameters of bilinear forms $\mathcal{A}_h(\cdot,\cdot)$, and $m_h(\cdot,\cdot)$
but independent of mesh size $h$ and time-step $\Delta t$ such that the following estimation holds  
\begin{equation*}
\begin{split}
\|\widetilde{U^{n}_h}-u^n\|_{2,\Omega} &\leq C \left( \Big \|\frac{\psi^1-\psi^0}{\Delta t} \Big \|_{0,\Omega}
+  \|\psi^0 \|_{2,\Omega}+  \|\psi^1 \|_{2,\Omega} \right )+ C h^{k-1} \Big(|u_0|_{k+1,\O}+|g|_{L^{\infty}(0,T;H^{k+1}(\Omega))} \\
& \quad+\|D_{tt}u\|_{L^2(0,T;H^{k+1}(\Omega))}+\|D_{t}u\|_{L^2(0,T;H^{k+1}(\Omega))}+\|u\|_{L^{\infty}(0,T;H^{k+1}(\Omega))}\Big)\\
& \quad+C\Delta t^2 \Big(\|D_{ttt} u\|_{L^{\infty}(0,T;L^2(\Omega))} + \|D_{tt} u\|_{L^{\infty}(0,T;L^2(\Omega))}+\|D_{t} u\|_{L^{\infty}(0,T;L^2(\Omega))} \Big). 
\end{split}
\end{equation*}
\end{theorem}
\begin{proof}
The proof follows analogous arguments with a minor modification of estimates of  nonlocal  term.
In fact, following \cite[Theorem~5.3]{AN-MCS2020} and the errors due to nonlocal term can be bounded as
\begin{equation*}
\begin{split}
&-\Big[S \sum_{E \in \O_h}\int_{E}|\Pi^{k-1}_E D_x \widetilde{U_h^{n-2}}|^2-P \Big] a_h^x(R_hu^{n},v_h)+\Big[S \sum_{E \in \O_h}\int_{E} |D_xu^{n}|^2-P \Big] a^x(u^{n},v_h)\\ 
& \leq C~ \Big(h^{k}~|u^{n}|_{k+1,\O}~\|D_x v_h\|_{0,\O}+h^k~|u^{n-2}|_{k+1,\O}~\|D_x v_h\|_{0,\O}+\|\widetilde{\psi^{n-2}}\|_{2,\O}~\|D_x v_h\|_{0,\O} \\
& \quad +\Delta t \|D_t u^{n-1}\|_{0,\O} \|D_x v_h\|_{0,\O} \Big ).
\end{split}
%\label{fulldis:linear:proof}
\end{equation*}
 Proceeding same as Theorem~\ref{fulldis:mainth}, we obtain the intended result.
\end{proof}

\section{Numerical experiments}
\label{Numerics:VEM}
In this section, we would like to demonstrate the performance of the proposed method
for the lowest order $C^1$ conforming virtual element space, i.e. for polynomial of degree $k=$ 2.
We have studied two test cases where the first example deals with manufactured solution and
another case which is focused on more realistic example. We have computed the numerical
solutions on different type of meshes including smoothed Voronoi, regular polygons, non-convex,
distorted square and square meshes (see Figure~\ref{fig:mesh}). The time dependent nonlocal plate equation models
the deformation of bridges. The nonlocal nonlinearity appears in the model problem~\eqref{modl_prob:1}-\eqref{modl_prob:5}
due to the stretching of the plate in the $x$-direction. The function~$g$
represents the vertical load over the plate. In \cite{CCCHSV-jFI2020},
authors have focused to study the decay of the energy with the time progression
surveying some practical examples with or without the presence of an external
load function $g$. However, in a fully discrete form, the model problem reduces
to a system  of nonlinear equations, which has to be solved numerically.
Traditional techniques based on FEM are expensive, since the
it requires $C^1$ elements and the presence of the nonlocal term destorys
the sparse structure of the Jacobian of the nonlinear system (Newton's Method).
To avoid these difficulties, we have introduced a new independent variable and
maintained the sparse structure of the Jacobian as shown in Figure~\ref{fig:examFig1jacstruc}.
All the Jacobian matrices are computed for meshes of 16$\times$16 elements for square,
distorted square, regular polygons and non-convex meshes and for Voronoi mesh with mesh size $h=1/10$.

\paragraph{Aspect of implementation of projection operator $\boldsymbol{\Pi^{k-1}_E}$.}
In \eqref{full:discrete:model}, we have discretized the nonlocal term using
the projection operator $\Pi^{k-1}_E$ which is computable form the degrees
of freedom $D1-D5$ for any order of $k \geq 2$. Further, we explain
briefly the computation of the projection operator $\Pi^{k-1}_E$ for arbitrary order $k \geq 3$ as follows
\begin{equation}\label{computyhgf}
\begin{split}
\int_E \Pi^{k-1}_E D_x \phi~ q& =\int_E D_x \phi~ q \qquad \qquad \qquad \qquad \forall q \in \mathbb{P}_{k-1}(E) \\
&=-\int_E \phi ~D_x q +\int_{\partial E} \phi~ n_x q, 
\end{split}
\end{equation}
where $n_x$ is a $x$ component of unit outward normal vector. The term $\int_{\partial E} \phi\  n_x ~q$
consists of an integral of a polynomial of degree $2k-1$. Since the virtual function $\phi$ is polynomial of degree $k$ on edge 
$e \subset \partial E$ and is explicitly computable from the degrees of freedom
associated with the discrete space. The another function $q$ is a known polynomial of degree $k-1$ and
hence the integration is fully computable from degrees of freedom $D1-D4$.
For the case $k=$ 2, the virtual function $\phi|_{e} \in \mathbb{P}_3(e)$;
consequently, we have four unknown coefficients which can be computed
from the four DoFs $D1-D2$.

\paragraph{Stabilization bilinear forms.}
To complete the choice of the VEM scheme for $k=2$,
we had to fix the forms $S_{\Delta}^{E}(\cdot,\cdot)$ and $S_{m}^{E}(\cdot,\cdot)$ (cf. Section~\ref{bili-formdfg}).
In particular, we have considered the form
\begin{align*}
S^{E}(\omega_h,v_h):=\sum\limits_{i=1}^{N_{E}}
[\omega_h(\Xi_i)v_h(\Xi_i)
+\nabla \omega_h(\Xi_i)\cdot\nabla v_h(\Xi_i)] & \qquad \forall \omega_h,v_h\in \ZK,
\end{align*}
where $\Xi_1,\ldots,\Xi_{N_{E}}$ are the vertices
of $E$.
Thus, we take $S_{\Delta}^{E}(\cdot,\cdot)$ and $S_{m}^{E}(\cdot,\cdot)$
in terms of $S^{E}(\cdot,\cdot)$, properly scaled
(see \cite{BM13} for further details).

\begin{figure}[!t]
  \begin{center}
    \begin{tabular}{cc}
      \includegraphics[width=0.45\textwidth]{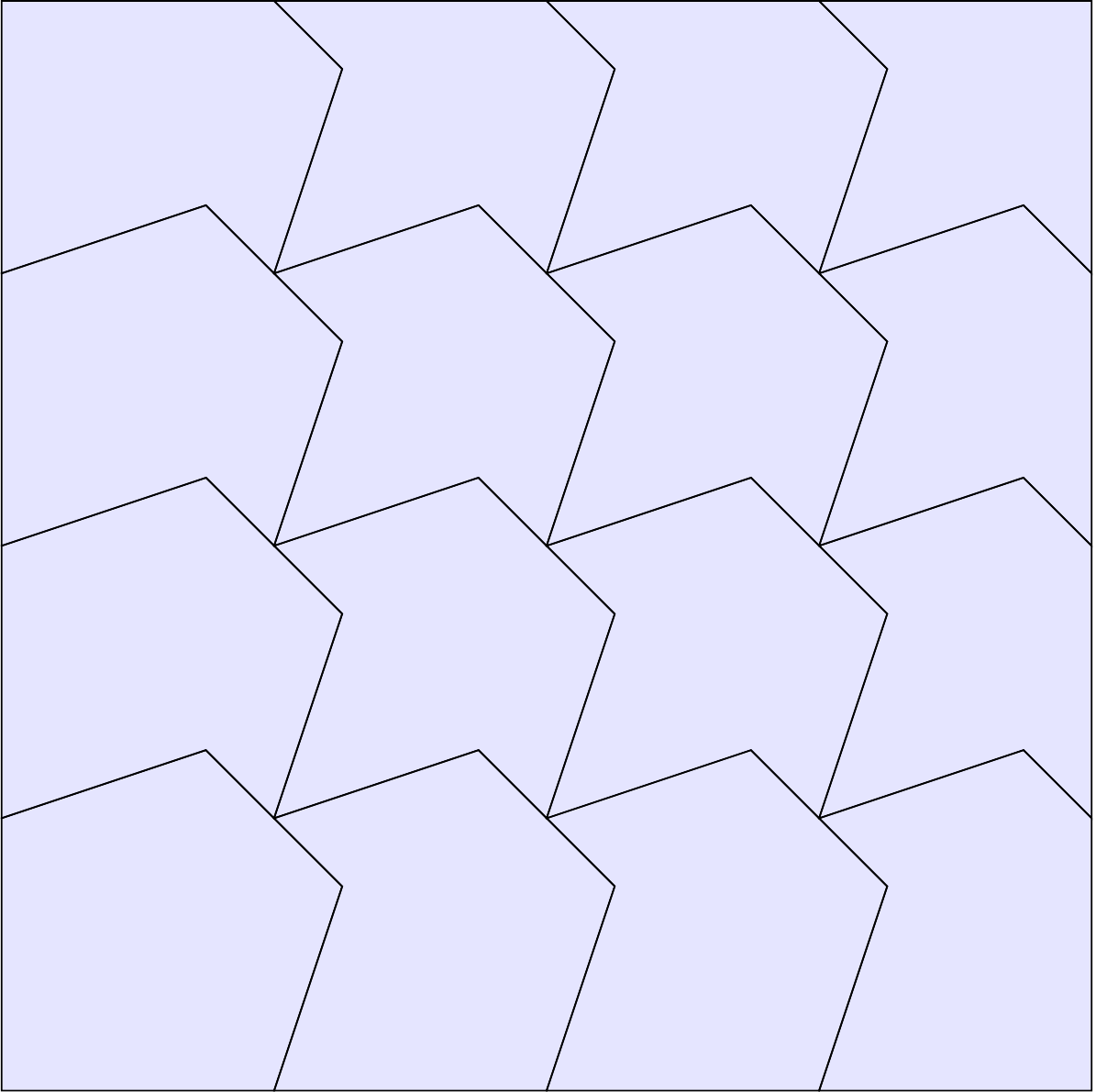} & \qquad
      \includegraphics[width=0.45\textwidth]{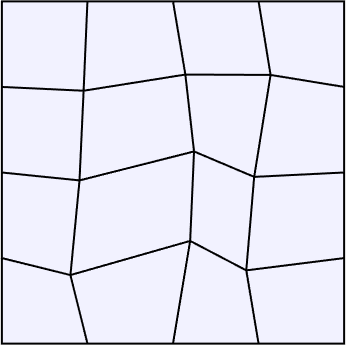} \\[0.5em]
      (a) Non-convex & (b)Distorted square \\ [1em]
      \includegraphics[width=0.45\textwidth]{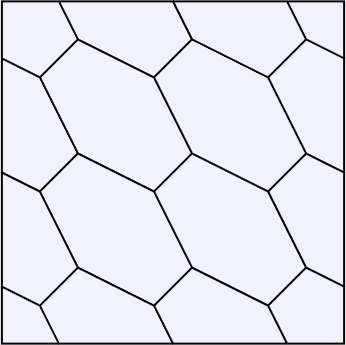} & \qquad
      \includegraphics[width=0.45\textwidth]{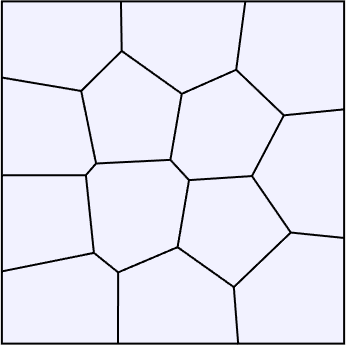} \\[0.5em]
      (c) Regular polygons & (d) Smoothed Voronoi
    \end{tabular}
  \end{center}
  \vspace{-0.25cm}
  \caption{A schematic representation of different discretizations
    employed in this study.}
  \label{fig:mesh}
\end{figure}

\subsection{Example~1}

First, we have considered a clamped plate with a manufactured solution.
We consider the model problem~\eqref{modl_prob:1} with the exact solution
$u(x,y,t):=\sin(\pi t)(x-x^2)^2(y-y^2)^2$ in $\O:=(0,1)^2$
and $u=\partial_{\bold{n}} u=0$ on $\partial\O$.
The damping coefficient $\delta$ is chosen as $1$ and the pre-stressing
constant $P$ and the elasticity of the material $S$ are chosen as $10^{-3}$ and $10^{-5}$, respectively.
The final time $T$ is chosen as $1/2$. Initial guess $U_h^0$ is considered as zeros
and $U_h^1$ is computed using the finite difference formula $\frac{U_h^1-U_h^0}{\Delta t}=\frac{d u}{dt}(0)$.
The time steps are chosen as sufficiently small to achieve the optimal rate of convergence in space variable.
The errors are computed using the formula
\begin{equation}
\mathcal{E}_2(u)=\Big (\sum_{E \in \O_h}|u^n-\PiK U_h^n|_{2,E}^2 \Big )^{1/2}.
\end{equation}
However, one can also compute the errors using the relative error formula such as
\begin{equation}
\mathcal{E}_{\text{rel}}(u):=\frac{\mathcal{A}_h(u^n-U_h^n,u^n-U_h^n)}{\mathcal{A}_h(u^n,u^n)}.
\end{equation}
We have displayed the solutions with the family of meshes containing
4$\times$ 4, 8$\times$8, 16$\times$16, 32$\times$32, and 64$\times$64 elements.
In Figure~\ref{fig:examFig1jacstruc}, we have shown the sparse structure of Jacobians
and  condition numbers of the Jacobians are shown in Figure~\ref{fig:exam1_jac},
where it is seen that the condition numbers increase as $\approx O(h^4)$.
Further, the nonlocal term is discretized using the $L^2$ projection
operator $\Pi^{k-1}_E$ on each element $E$ (cf. \eqref{computyhgf}). Thus,
we can compute the term $\Pi^{k-1}_E D_x\phi_i$.
In the discretization of $a^x(\cdot,\cdot)$, we have considered only polynomial
part avoiding non-polynomial part or stabilization part. However, we have proved
theoretically that the fully-discrete scheme~\eqref{full:discrete:model}-\eqref{full:discrete:initial}
is well posed and converges optimally in both space and time variables.
In particular, let $(\bold{A}^x)_{ij}=\sum_{E \in \O_h}\int_{E} \Pi^{k-1}_E D_x \phi_i~\Pi^{k-1}_E D_x \phi_j {\rm d}E$,
then the nonlocal term could be computed as
$\sum_{E \in \O_h}\int_{E} |\Pi^{k-1}_E D_x U_h^n|^2 {\rm d} E= [\boldsymbol{\eta}^n] \bold{A}^x [\boldsymbol{\eta}^n]^T$,
where $[\boldsymbol{\eta}^n]$ is the coefficient vector defined in \eqref{discrete:sol:explicit}.
In Figure~\ref{fig:exam1L2conv}, we have displayed the convergence behaviour for
different type of meshes for nonlinear scheme. It is inferred that the proposed
framework yields optimal convergence in the $H^2$-norm.

On the other hand, we mention that the linearized scheme~\eqref{linear:discrete:model}-\eqref{linear:discrete:initial}
where the nonlinear term is computed at previous step. Thus, the fully-discrete linearized scheme
reduces to system of linear equations, we can employ any linear solver to compute
the system of equations and we dare to leave the scheme without verifying experimentally.

\begin{figure}[htpb]
\centering 
\subfigure[Distorted square]{\includegraphics[scale=0.5]{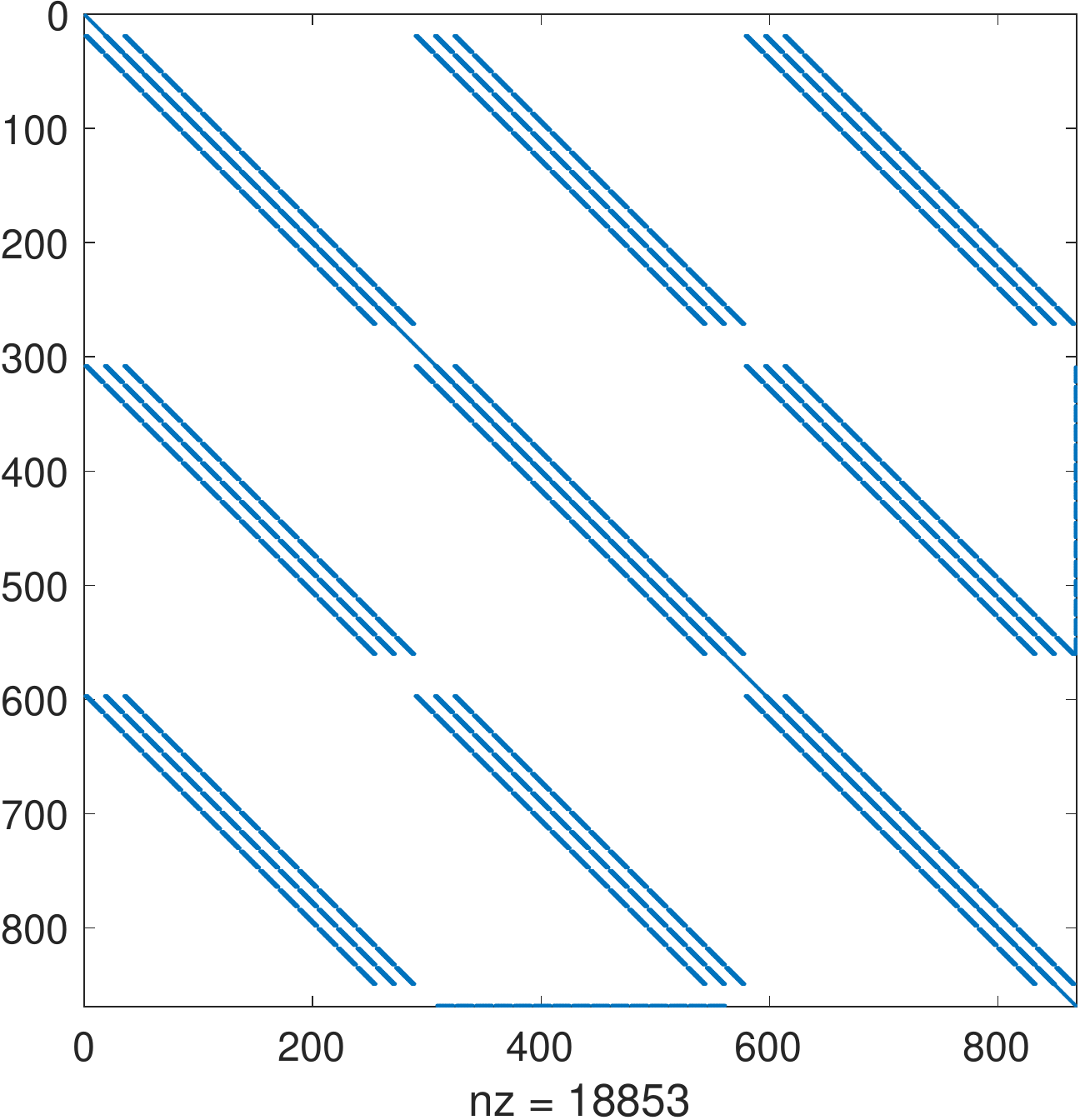}} 
\subfigure[Non-convex]{\includegraphics[scale=0.5]{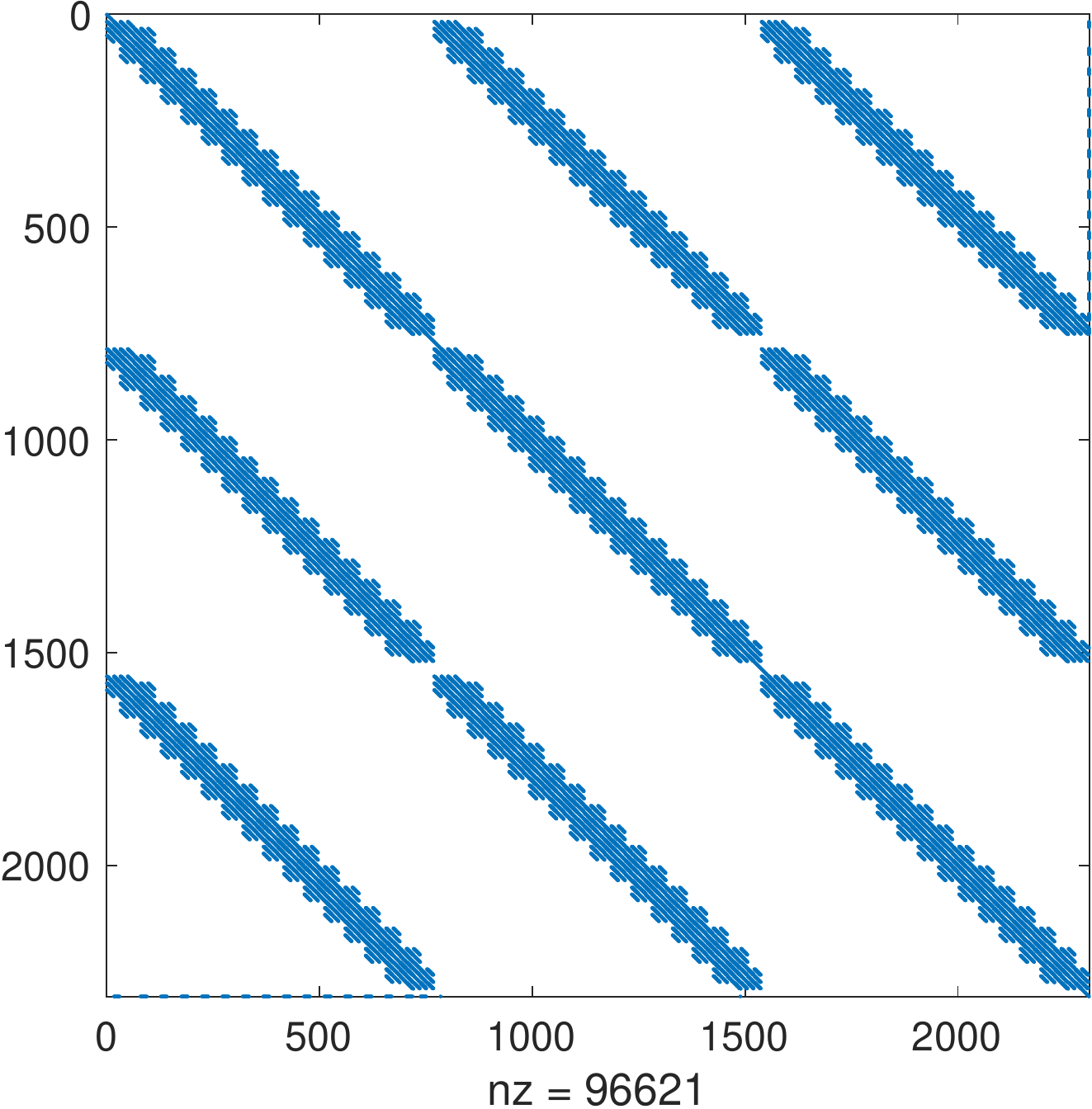}}\\
\subfigure[Square]{\includegraphics[scale=0.5]{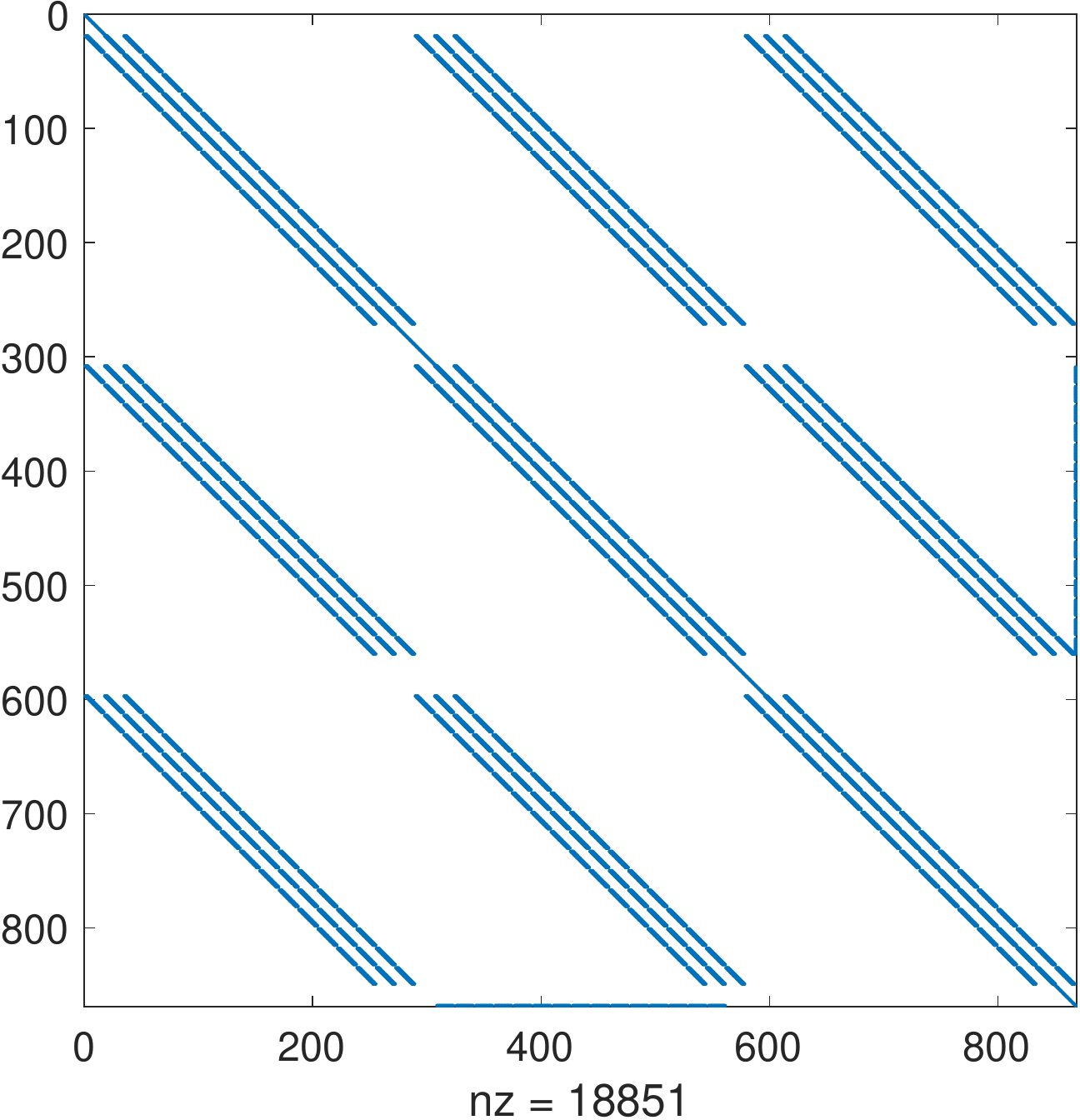}} 
\subfigure[Uniform polygon]{\includegraphics[scale=0.5]{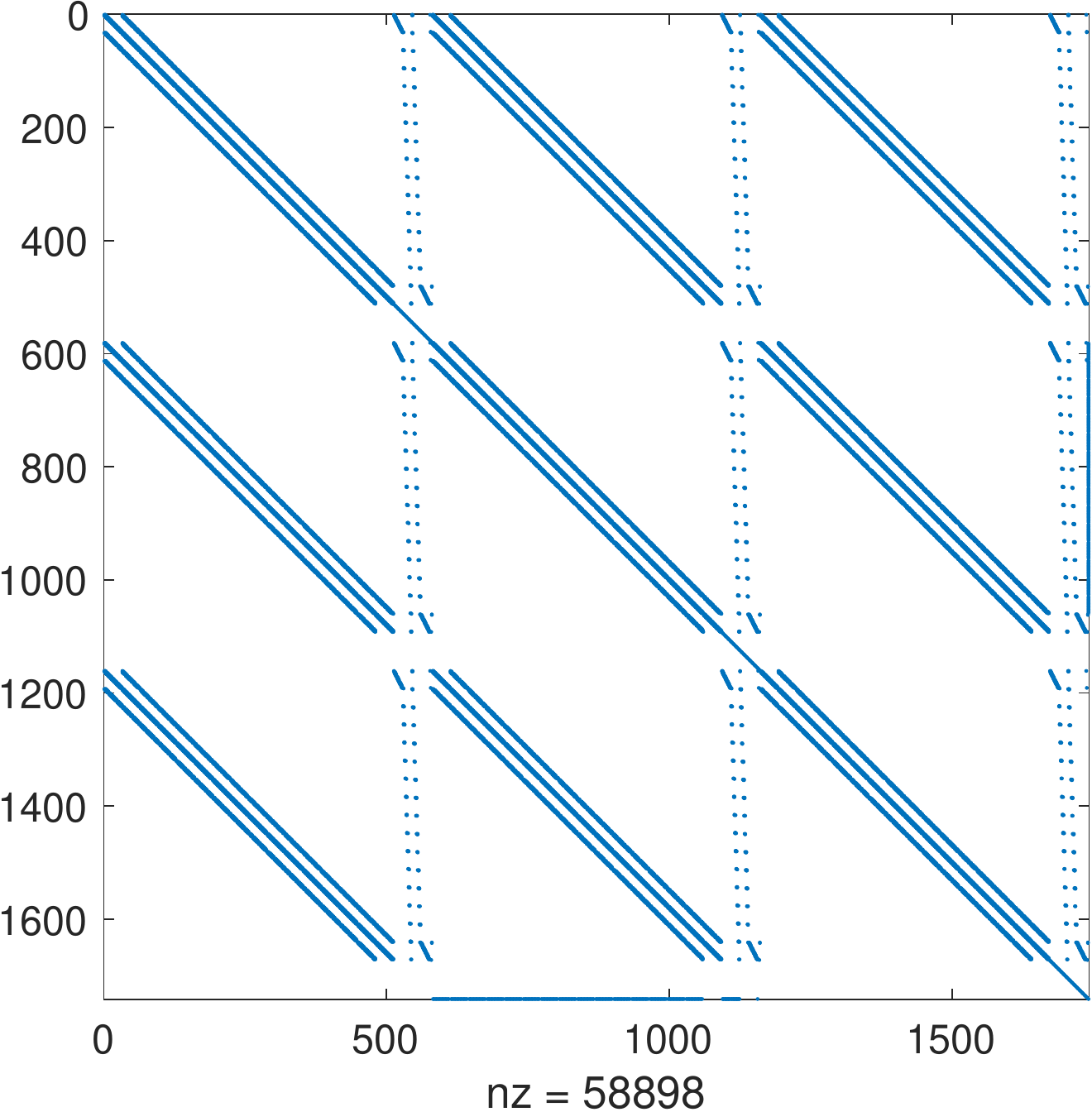}}
\subfigure[Voronoi]{\includegraphics[scale=0.5]{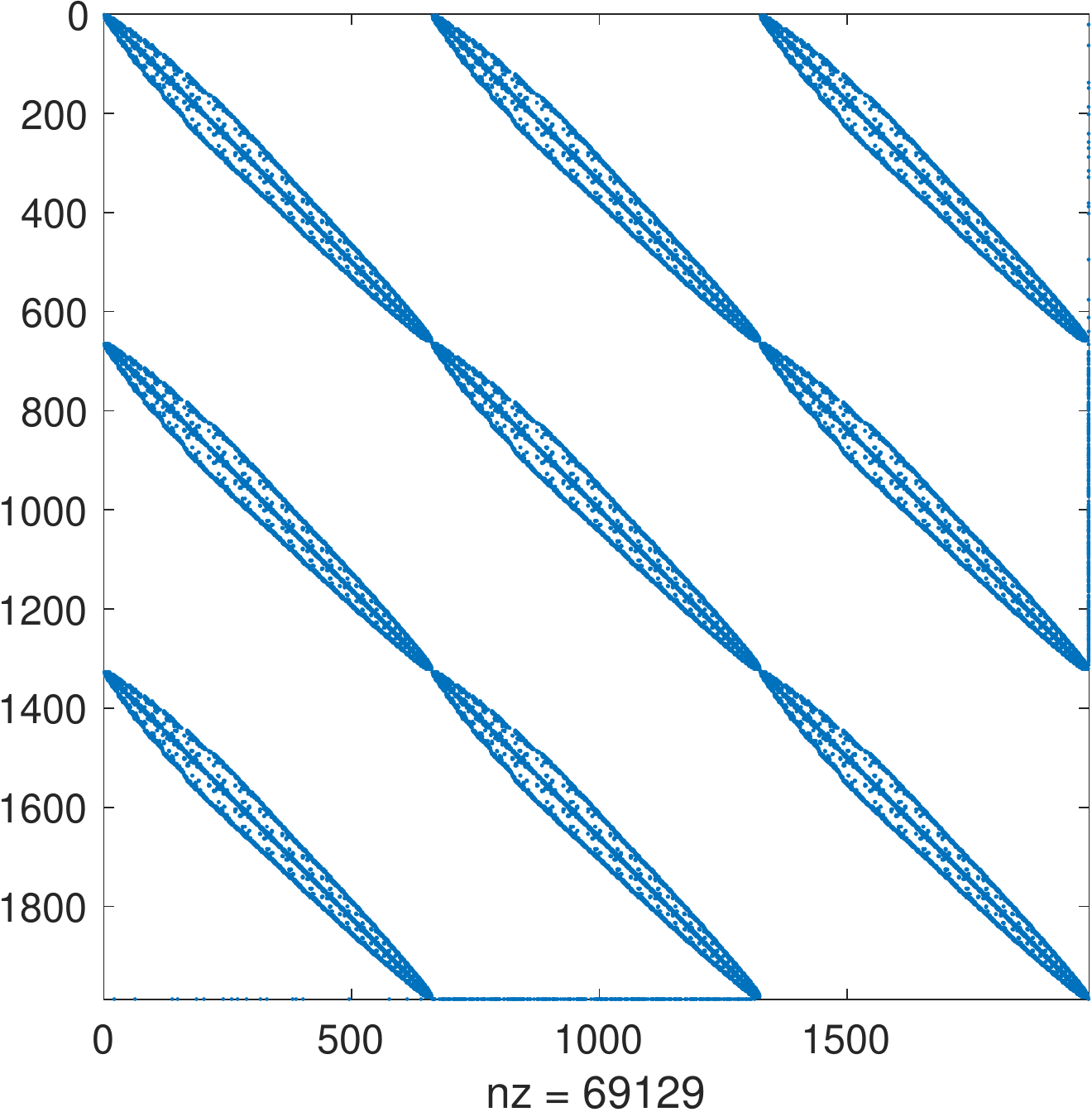}}
\caption{Example 1: Sparsity of Jacobian for different discretization considered in this analysis.} 
\label{fig:examFig1jacstruc}
\end{figure}

\begin{figure}[htpb!]
\centering 
\subfigure[]{\includegraphics[scale=0.65]{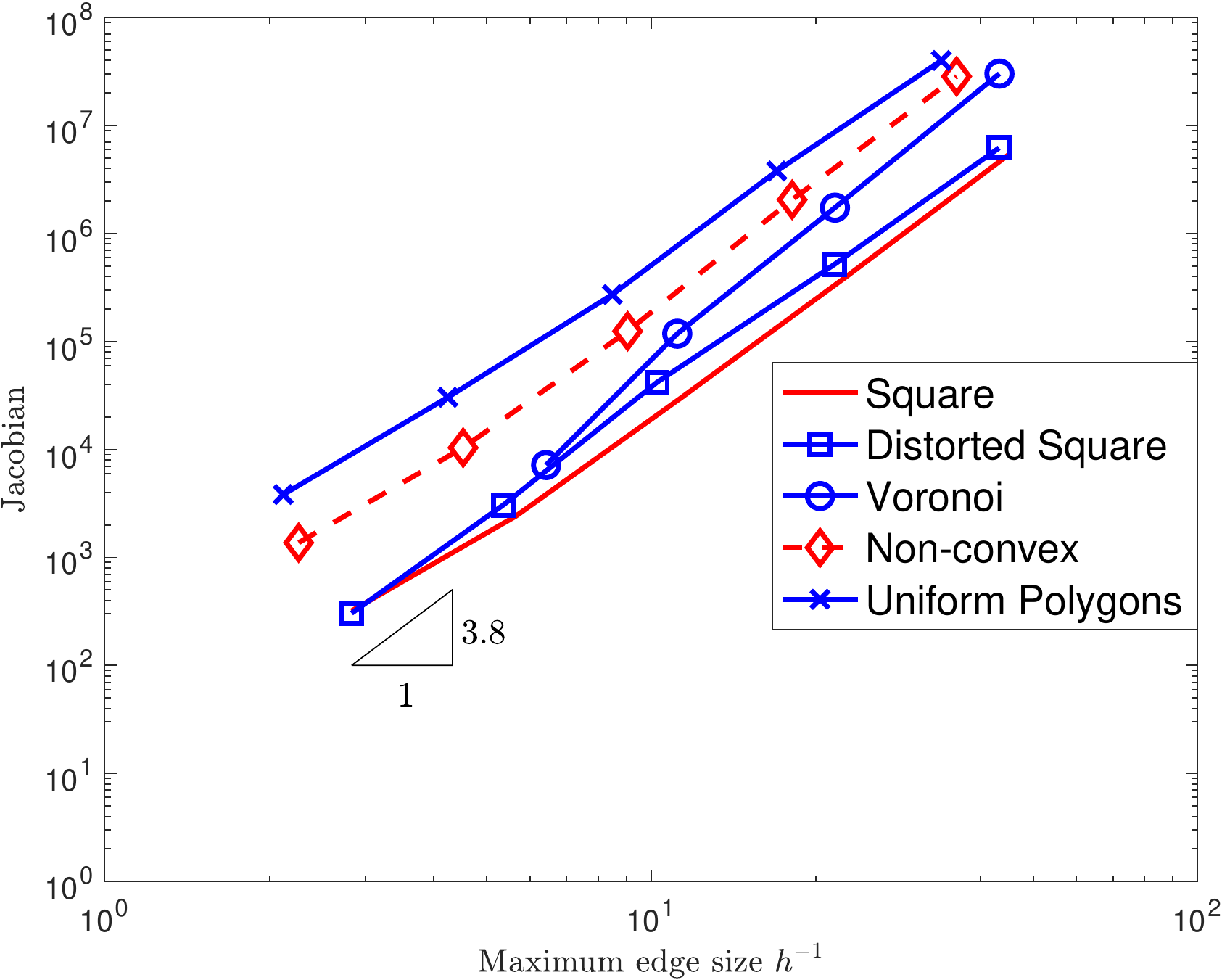}\label{fig:exam1_jac}}
\subfigure[]{\includegraphics[scale=0.65]{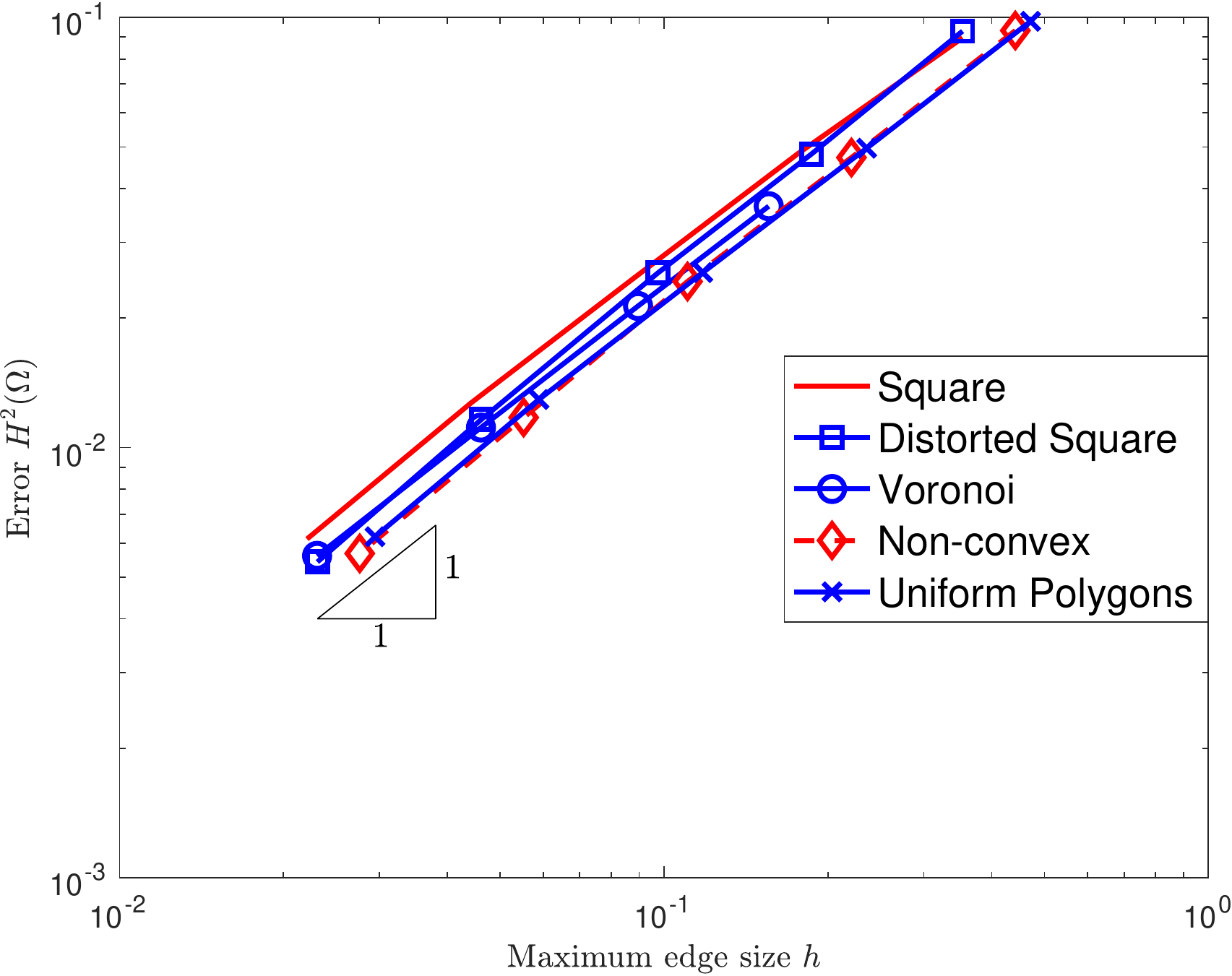}\label{fig:exam1L2conv}}
\caption{Example 1: (a) conditioning number of the Jacobian matrix and (b) convergence of the  error in the $H^2$ norm with mesh refinement.} 
\label{fig:examFig1}
\end{figure}

\subsection{Example~2}

In this section, we have borrowed a more realistic example that models the deformation of a real bridge
\cite[Numerical Example~5.4]{CCCHSV-jFI2020}.  We consider $g=0$, $P=10^{-3}$, $S=10^{-5}$, $\sigma=0.2$.
The computational domain is considered as $[0,\pi]\times [-\ell,\ell]$, where $\ell=\pi/150$.
The initial guess $U^0_h$ is chosen as the solution of the following stationary problem:
\begin{align*}
& \Delta^2 u=50 \sin(2x) \quad \quad \text{in} ~\O,\\
&u(0,y)=D_{xx} u(0,y)=u(\pi,y)=D_{xx} u(\pi,y)=0 \quad y \in (-\ell,\ell),\\
& D_{yy} u(x,\pm \ell)+\sigma D_{xx}u(x,\pm \ell)=0, \quad  x \in (0,\pi),  \\
& D_{yyy} u(x,\pm \ell)+(2-\sigma) D_{xxy}u (x,\pm \ell)=0, \quad x \in (0, \pi).
\end{align*}
Further, we choose the immediate next approximation $U_h^1=U_h^0$ ($\omega_0=0$). 
\begin{figure}[htpb!]
\centering
\includegraphics[scale=0.45]{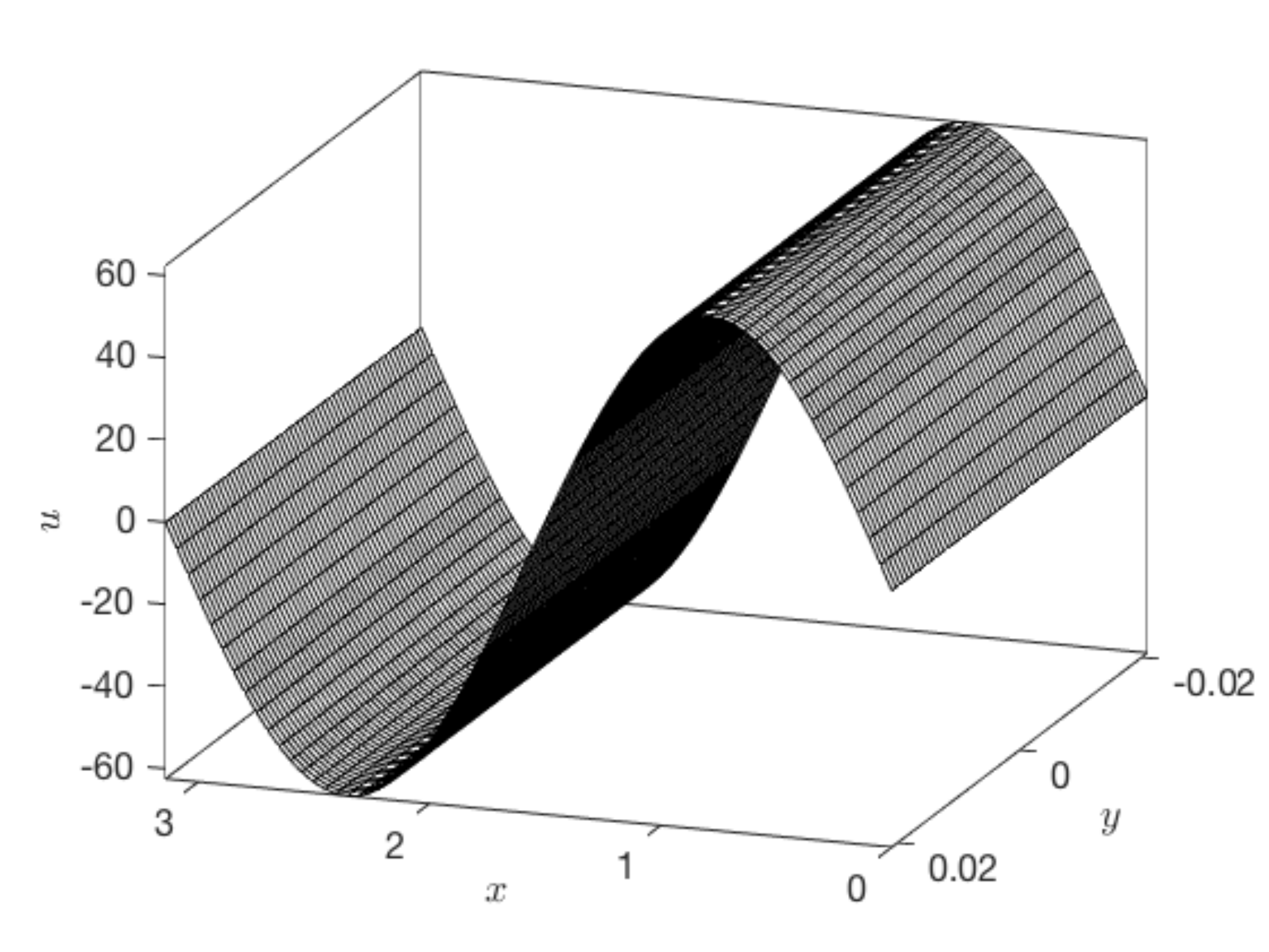}
\caption{Example 2: Initial guess $U^0_h$}
\label{fig:example2_initial_guess}
\end{figure}
Figure~\ref{fig:example2_initial_guess}, shows the initial guess $U^0_h$
which is basically the initial deformation at time $t=0$.  It is noted
that the damping coefficient $\delta(x,y)$ plays a significant role in
decaying the energy as time $t$ goes to infinity, which should be positive
($>0$) in the neighbourhood of $\widetilde{\O}:=(0,10h)\cup (\pi-10h,\pi) \times (-\ell, -\ell+5h) \cup (\ell-5h,\ell)$,
and zero in $\O \setminus \widetilde{\O}$. We choose 
\begin{equation*}
\delta(x,y)=
\begin{cases}
& 1, \quad  \quad (x,y) \in \widetilde{\O} \\
&0, \quad \quad (x,y) \in \O \setminus \widetilde{\O}. 
\end{cases}
\end{equation*}
For the computation, the final time $T$ is chosen as $0.01$. We have computed numerical
results at final time and the convergence of the $H^2$ errors with mesh refinement
are shown in Figure~\ref{fig:examFig2}. Time-step is taken as $\Delta t=O(h)$.
Further, the rate of convergence of the numerical solution is in accordance with
the theory as proved in Theorem~\ref{fulldis:mainth}. In addition, we also would
like to study the decay of energy as it is analyzed in \cite[Theorem~3.7]{CCCHSV-jFI2020}.
The energy is defined as
\begin{equation}
E_u(t):=\frac{1}{2} \|D_t u(t)\|_{L^2(\Omega)}^2+\frac{1}{2} \|u(t)\|^2_{H^2_{\ast}(\Omega)}-\frac{P}{2} \|D_x u(t)\|^2_{L^2(\O)}+\frac{S}{4} \|D_x u(t)\|^4_{L^2(\O)}, 
\end{equation}
where $t\geq 0$. It is estimated that $$ E_u(t) \leq S \Big( \frac{t}{T_0}-1\Big) \quad \forall t \geq T_0 >0,$$
and $\lim_{t \rightarrow \infty} S(t)=0$. 
In discrete version, the energy is defined by using the virtual element discrete solution as 
\begin{equation}
\begin{split}
\widetilde{E}_{U_h^n}(t)&:=\frac{1}{2} \sum_{E\in \O_h} m_h^E \left(\frac{U_h^n-U_h^{n-1}}{\Delta t},\frac{U_h^n-U_h^{n-1}}{\Delta t} \right )+\frac{1}{2} \sum_{E \in \O_h} \calA_h^E( U_h^n,U_h^n)\\
&\quad-\frac{P}{2} \sum_{E \in \Omega_h} \|\Pi^{k-1}_E D_x U_h^n\|^2_{0,E} +\frac{S}{4} \sum_{E \in \Omega_h}\|\Pi^{k-1}_E D_x U_h^n\|^4_{0,E},
\end{split} 
\end{equation}
We have post processed the results to recover the norm such as $ \sum_{E \in \O_h}\|\Pi^{k-1}_E D_x U_h^n\|_{0,E}^2$,
using the following discrete bilinear form:
\begin{equation}
\sum_{E \in \O_h} \|\Pi^{k-1}_E D_xU_h^n\|_{0,E}^2:=a_h^x(U_h^n,U_h^n),
\end{equation}
or, in matrix form,$$ \sum_{E \in \O_h} \|\Pi^{k-1}_E D_xU_h^n\|_{0,E}^2:=(\boldsymbol{\eta}^n)^T \bold{A}^x (\boldsymbol{\eta}^n),$$ 
where $\boldsymbol{\eta}^n$ and $\bold{A}^x$ are defined in Section~\ref{plate_implement}.
The energy is computed on square mesh for $16 \times 16$ elements with very small time step $\Delta t=1/1000$.
In Figure~\ref{fig:examp_energy}, we have plotted Energy versus final time $T$
and it is clearly observed that as $T \rightarrow \infty $, the energy $\widetilde{E}_{U_h^n} \rightarrow 0$. From \eqref{fig:examp_energy}, we deduce that the bridge reaches immobilized condition at time $T=5$ with initial deformation as defined in Example~2.
 Finally, we would also like to conclude the discussion by dissecting the Jacobian without introducing new variable as $\xi$. Recollecting \eqref{jacobian}, it can be observed that $(\bold{J})_{ij} \neq 0$ for $i \neq j$. The Jacobian is displayed in Figure~\ref{fig:examp2_jac}. The Jacobian is computed for the full matrix, without introducing the additional variable on a square mesh for Example~2. It is observed that the number of non-zeros is greater in this case when compared to the case with an additional variable. Further, it is also observed that the Newton iterations take relatively more iterations without inviting the new additional variable and the number of non-zeros in Jacobian increases as the mesh size, $h$ approaches 0. It is clear from this discussion, that the presented framework is advantageous in spite of having an additional variable.

\begin{figure}[htpb]
\centering 
\includegraphics[scale=0.65]{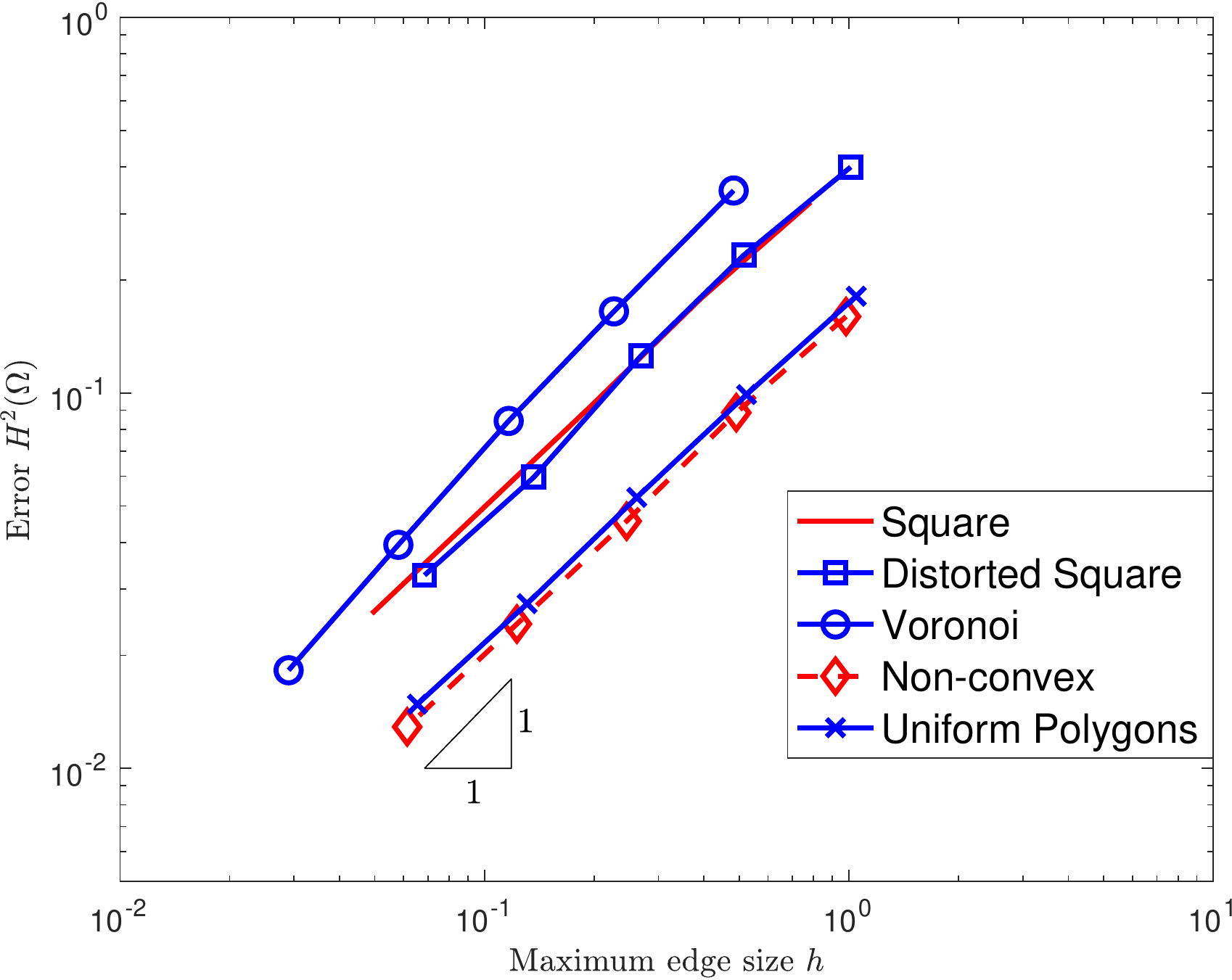}
\caption{Example 2: convergence of the error $H^2(\Omega)$ with mesh refinement for different types of discretizations.} 
\label{fig:examFig2}
\end{figure}

\begin{figure}[htpb]
\centering 
\includegraphics[scale=0.65]{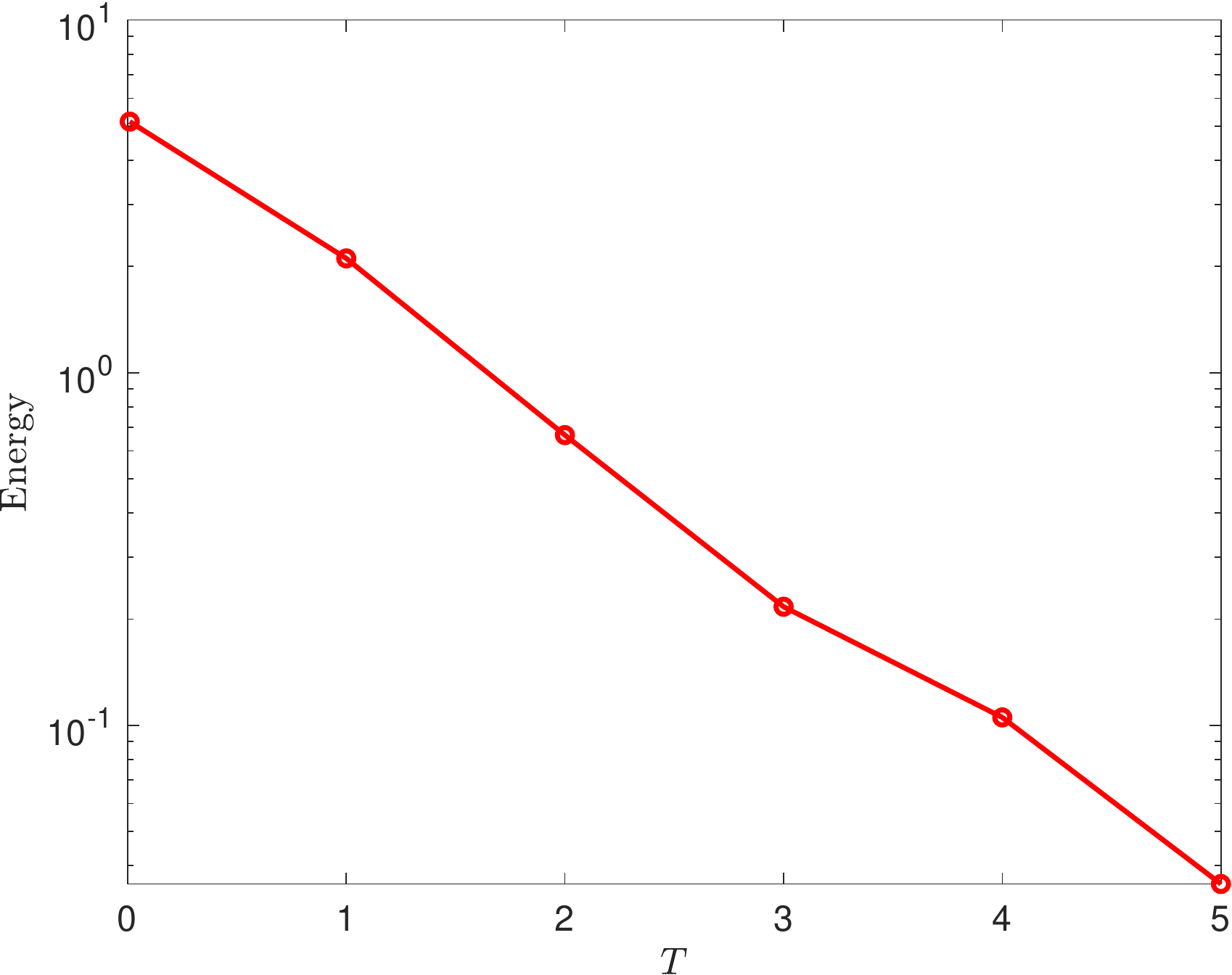}
\caption{Example 2: Decay of energy with respect to time.} 
\label{fig:examp_energy}
\end{figure}

\begin{figure}[htpb]
\centering 
\includegraphics[scale=0.65]{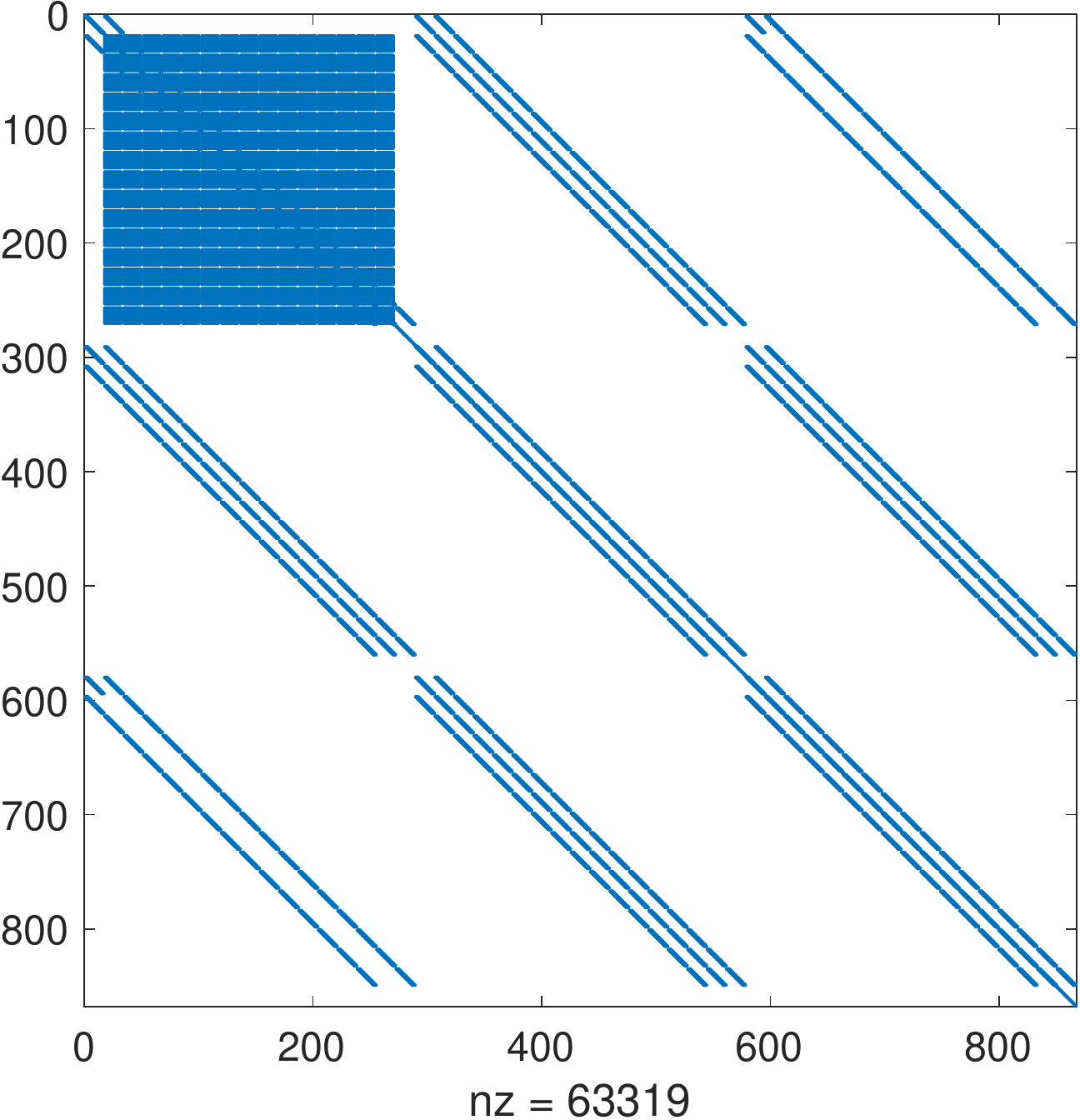}
\caption{Example 2: Jacobian computed on square mesh without introducing the new variable. It is clear that the number of non-zeros are more than the one proposed in this paper.} 
\label{fig:examp2_jac}
\end{figure}

\section{Conclusion}
In this article, we have proposed a numerical technique to solve the time dependent
nonlocal plate problem which models deformation of bridge. In order to discretize
time dependent part, we need to compute the $L^2$ projection operator and accordingly,
we have modified the VEM space which allows the full computation of the $L^2$
projection operator. Further, the model problem deals with nonlocal nonlinearity
which spoils the sparse structure of the Jacobian matrix and consequently computational cost.
We have addressed this difficulty by introducing independent variable and retrieve the
sparse structure of the Jacobian. Wellposedness of the fully discrete scheme and {\it a priori}
error estimates are derived in $H^2$ norm. Finally, we explore the workability of the numerical
technique by examining two benchmark examples including a problem with manufactured
solution and clamped boundary condition and another cases are prototype of more
realistic bridge modelling  without external vertical force functions.
Also, we have studied the potential and kinetic energy associated with deformation
of bride and plotted against time to demonstrate the uniform decay of energy as
time goes to infinity as claimed in \cite{CCCHSV-jFI2020}.

\section*{Acknowledgements}
The first author was partially supported by the National Agency for Research and Development, ANID-
Chile through FONDECYT Postdoctorado project 3200242.
The second author was partially supported by the National Agency for Research and Development,
ANID-Chile through FONDECYT project 1180913, by projects ACE210010 and
Centro de Modelamiento Matem\'atico FB210005
and by DIUBB through project 2120173 GI/C.

\end{document}